\documentclass{amsart}
\usepackage{amsmath}
\usepackage{amsthm}
\usepackage{amsfonts,mathrsfs}
\usepackage{amssymb}
\usepackage{graphicx}
\usepackage{enumitem}
\usepackage{color}
\usepackage{verbatim}
\theoremstyle{plain}
\newtheorem{theorem}{Theorem}[section]
\newtheorem{lemma}[theorem]{Lemma}
\newtheorem{remark}[theorem]{Remark}

\theoremstyle{definition}

\newtheoremstyle{TheoremNum}
{\topsep}{\topsep}              %%% space between body and thm
{\itshape}                      %%% Thm body font
{}                              %%% Indent amount (empty = no indent)
{\bfseries}                     %%% Thm head font
{.}                             %%% Punctuation after thm head
{ }                             %%% Space after thm head
{\thmname{#1}\thmnote{ \bfseries #3}}%%%

\newcommand{\fq}{\mathbb{F}_{q}}
\newcommand{\fqn}{\mathbb{F}_{q^n}}
\newcommand{\F}{\mathbb F}

\newcommand{\Z}{\mathbb Z}

\newcommand{\PG}{\mathrm{PG}}

\newcommand{\Aut}{\mathrm{Aut}}
\newcommand{\PGL}{\mathrm{PGL}}

\newcommand{\PGamL}{\mathrm{P\Gamma L}}
\newcommand{\GamL}{\mathrm{\Gamma L}}
\newcommand{\GL}{\mathrm{GL}}

\newcommand{\RN}[1]{%
	\textup{\uppercase\expandafter{\romannumeral#1}}%
}
\newcommand{\CFA}{\mathscr{F}_6}
\newcommand{\CFB}{\mathscr{F}_8}
\newcommand{\CF}{\mathscr{F}}

\def\zhou#1 {\fbox {\footnote {\ }}\ \footnotetext { From Yue: {\color{red}#1}}}

\def\chen#1 {\fbox {\footnote {\ }}\ \footnotetext { From Tang: {\color{blue}#1}}}

\begin{document}
	\title[Equivalence and Automorphism groups of two families of MSLS]{Equivalence and automorphism groups of two families of maximum scattered linear sets}
	\author[W. Tang]{Wei Tang\textsuperscript{\,1}}
	\author[Y. Zhou]{Yue Zhou\textsuperscript{\,1}}
	\address{\textsuperscript{1}Department of Mathematics, National University of Defense Technology, 410073 Changsha, China}
%	\address{\textsuperscript{$\dagger$}Corresponding author}
	\email{yue.zhou.ovgu@gmail.com}
	\thanks{This is an English translation of the original Chinese version published in Scientia Sinica Mathematica. DOI: 10.1360/SSM-2022-0073}
%	\date{\today}
	\begin{abstract}
		Linear set in projective spaces over finite fields plays central roles in the study of blocking sets, semifields, rank-metric codes and etc. A linear set with the largest possible cardinality and  the maximum rank is called maximum scattered. Despite two decades of study, there are only  a few number of known maximum scattered linear sets in projective lines, including two families constructed in \cite{csajbok_classes_2018}, and \cite{csajbok_newMSLS_2018,marino_classes_2020}, respectively. This paper aims to solve  the equivalence problem of the linear sets in each of these families and to determine their automorphism groups.
	\end{abstract}
	\maketitle
	\section{Introduction}
	
	The concept of \emph{linear set}   was introduced by Lunardon \cite{lunardon_normal_1999}, which is  natural generalization of subgeometrie. In the pasting two decades after this work, linear sets have been intensively investigated and applied to construct and characterize various objects in finite geometry and coding theory, including blocking sets, two-intersection sets, translation spreads of the Cayley generalized Hexagon, translation ovoids of polar spaces, semifields and rank-metric codes. We refer to \cite{bartoli_maximum_2018,lavrauw_field_reduction_2015,polverino_linear_2010,polverino_connections_2020,sheekey_new_2016} and the references therein.
	
	In this paper, we only concern with the linear sets in a projective line.	Let $\Lambda=\PG(V)=\PG(1,q^n)$, where $V$ is a $2$-dimensional vector space over $\F_{q^n}$. Then $\Lambda$ is a projective line. A point set $L$ of $\Lambda$ is said to be an \emph{$\F_q$-linear set} of rank $k$ if it is defined by the non-zero vectors of a $k$-dimensional $\F_q$-vector subspace $U$ of $V$, i.e.
	\[L=L_U:=\{\langle {\bf u} \rangle_{\mathbb{F}_{q^n}} : {\bf u}\in U\setminus \{{\bf 0} \}\}.\]
	For any \emph{$\F_q$-linear set} $L_U$ with rank $k$, it is not difficult to see that
	\begin{equation*}
		|L|\leq \frac{q^{k}-1}{q-1}.
	\end{equation*}
	If the equality holds,  $L_U$ is called a scattered linear set. In \cite{blokhuis_scattered_2000}, Blokhuis and Lavrauw show that a scattered linear set has rank at most $n$, and they call a scattered linear set of rank $k=n$ in $\Lambda$ \emph{maximum}.
	
	Under the action of $\PGamL(2,q^n)$, we can always assume that the linear set $L_U$ of rank $n$ in $\Lambda$ does not contain the point $\langle(0,1)\rangle_{\fqn}$. Consequently, $U$ and $L_U$ can be written as
		\[ U_f=\{ (x,f(x)) \colon x \in \fqn \}, \]
	and
	\[ L_f=\left\{ \langle (x,f(x)) \rangle_{\fqn} \colon x \in \fqn^* \right\}, \]
	where $f$ is a $q$-polynomial over $\fqn$, that is $f=\sum_{i=0}^{n-1}a_iX^{q^i}\in \fqn[X]$. It is not difficult to show that $L_f$ is scattered if and only if for any $z,y\in \F_{q^n}^*$ the condition
	\[\frac{f(z)}{z} = \frac{f(y)}{y}\]
	implies that $z$ and $y$ are $\F_q$-linearly dependent. Hence, $q$-polynomials satisfying this condition are called \emph{scattered} by Sheekey in \cite{sheekey_new_2016}.
	
	Up to now, there are only three families of maximum scattered linear sets in $\PG(1,q^n)$ for infinitely many $n$. We list the corresponding scattered polynomials over $\F_{q^n}$ below. 
	\begin{enumerate}[label=(\roman*)]
		\item $f=X^{q^s}$ with $\gcd(s,n)=1$, and $L_f$ is called the linear set of pseudoregulus type; see \cite{blokhuis_scattered_2000}.
		\item $f=X^{q^s}+\delta X^{q^{n-s}}$ 	with $n \geq 4$, $N_{q^n/q}(\delta)\notin \{0,1\}$ and $\gcd(s,n)=1$ which was found by Lunardon and Polverino \cite{lunardon_blocking_2001} and later generalized by Sheekey \cite{sheekey_new_2016}.
		\item $f= X^{q^s}+X^{q^{s(m-1)}}+h^{1+q^{s}}X^{q^{s(m+1)}}+h^{1-q^{2m-s}}X^{q^{2m-s}}$ with $q$ odd, $n=2m\geq 6$, $h^{q^m+1}=-1$ and $\gcd(s,n)=1$. This family was first introduced by Longobardi and Zanella in  \cite{longobardi_familyOfLinearMRD_2021} and later generalized to a big family in \cite{LMTZscatt} and \cite{neri_extending_arxiv}.
	\end{enumerate}
	
	In this paper, we are interested in several families of maximum scattered linear sets in $ \PG(1,q^6) $ and $ \PG(1,q^8) $. In the following we list their corresponding scattered polynomials.
	\begin{enumerate}[label=(\alph*)]
		\item $x^q+x^{q^3}+\delta x^{q^5}$ over $\F_{q^6}$ with $\delta^2+\delta=1$ and $q$ is odd; see \cite{csajbok_newMSLS_2018,marino_classes_2020}.
		\item $x^{q^{s}}+\delta x^{q^{s+m}}$ over $\F_{q^{2m}}$ with $m\in \{3,4\}$, $\gcd(s,m)=1$ and some conditions on $\delta$; see \cite{csajbok_classes_2018}.
	\end{enumerate}

	For convenience, we use $\CFA$ and $\CFB$ to denote the families of linear sets given in (b) for $m=3$ and $4$, respectively. Let $f_{\delta,m}$ stand for a member with parameter $(\delta,m)$ in $\CFA$ or $\CFB$.	
	
	For $m=3$,  it has been shown in \cite{csajbok_classes_2018} that for $q>4$ and $N_{{q^6}/{q^3}}(\delta)\neq 1$ we can always find $\delta\in\F_{q^2}^*$ such that $f_{\delta,3}$ is scattered. The exact number of such $\delta$ has been determined by Bartoli, Csajb\'ok and Montanucci in \cite{Bartoli_2020}.
	
	For $m=4$, it also has been proved in \cite{csajbok_classes_2018} that for odd $q$,  $N_{{q^8}/{q^4}}(\delta)\neq 1$ and $\delta^2=-1$,  $f_{\delta,4}$ is scattered. Moreover, Timpanella and Zini \cite{Timpanella} show that $L_{\delta,4}$ is maximum scattered if and only if $\delta^{1+q^4}=-1$, under the extra assumption that the odd prime power $q\geq 1039891$.
	
	For  two linear sets $L_U$ and $L_W$ of $\PG(1,q^n)$,  if there  is an element $\varphi$ in $\mathrm{P\Gamma L}(2,q^n)$ such that $L_U^{\varphi} = L_W$, then we say $L_U$ and $L_W$  are  \emph{$\mathrm{P\Gamma L}$-equivalent}. For given $q$ and $n$, it is obvious that there is only one maximum scattered linear set of  pseudoregulus type, because 
	\[
	\{ \langle(1, x^{q^{s}-1}) \rangle: x\in \F_{q^n}^* \}=	\{ \langle(1, x^{q-1}) \rangle: x\in \F_{q^n}^* \}
	\]
	for any integer $s$ satisfying $\gcd(s,n)=1$. The equivalence of different members in the Lunardon-Polverino family has been recently solved by the authors and Zullo in \cite{tang_equivalence_2022}. The equivalence for the scattered quadrinomials from the third infinite family has been partially determined in \cite{longobardi_familyOfLinearMRD_2021}.
	
	In this paper, we continue to consider the equivalence among the members in the families (a) and (b) defined on $\PG(1,q^6)$ and $\PG(1,q^8)$, respectively. Firstly, in Section 3, we prove that (a) contains only one element up to $\mathrm{P\Gamma L}$-equivalence, and we completely determine its automorphism group. Then, in Section 4, the necessary and sufficient conditions for the equivalence of the elements in $\CFA$ and $\CFB$ are given, and their automorphism groups are further determined. The main ideas for proving the above results are as follows: First, we transform the $\mathrm{P\Gamma L}$-equivalence problem  into the equivalence of image sets of the corresponding polynomial maps. By further analyzing and verifying some necessary conditions on the coefficients of these polynomials, and considering some properties of their adjoint polynomials, we can determine the equivalence between the linear sets and their automorphism groups.
	
	The rest part of this paper is organized as follows. In Section \ref{sec:2}, we introduce some preliminaries and auxiliary results for  the equivalence of maximum scattered linear sets. In Section \ref{sec:3}, we focus on the equivalence problem for the linear sets derived from the scattered polynomials given in (a) and determine their automorphism groups. In Section \ref{sec:4}, we turn to the equivalence problem of linear sets in $\CFA$ and $\CFB$. In section \ref{sec:5}, we briefly summarized the contents of this paper.
	
	\section{Preliminaries}\label{sec:2}
	
	A $q$-\emph{polynomial} (or a \emph{linearized polynomial}) over $\F_{q^n}$ is a polynomial of the following form
	\[f=\sum_{i=0}^{k} a_i X^{q^i},\]
	where $a_i\in \F_{q^n}$ and $k$ is a positive integer.
	
	We will denote by $\mathcal{L}_{n,q}$ the set of all $q$-polynomials over $\F_{q^n}$ with degree less than $q^n$. There is a bijection between $\mathcal{L}_{n,q}$ and  the $\fq$-linear map defined over $\F_{q^n}$.
	For more details on linearized polynomials we refer to \cite[Chapter 3, Section 4]{lidl_finite_1997}.
	
	Consider the non-degenerate symmetric bilinear form of $\F_{q^n}$ over $\F_q$ defined  by
	\begin{equation*}%\label{eq:bilform} 
	\langle x,y\rangle= \mathrm{Tr}_{q^n/q}(xy),
	\end{equation*}
	for every $x,y \in \F_{q^n}$ and where $\mathrm{Tr}_{q^n/q}(x)=x+\cdots+x^{q^{n-1}}$.
	The \emph{adjoint polynomial} $\hat{f}$ of the $q$-polynomial $\displaystyle f=\sum_{i=0}^{n-1} a_iX^{q^i}$ with respect to the bilinear form $\langle\cdot,\cdot\rangle$, i.e. the unique function over $\fqn$ satisfying
	\[ \mathrm{Tr}_{q^n/q}(yf(z))=\mathrm{Tr}_{q^n/q}(z\hat{f}(y)) \]
	for every $y,z \in \F_{q^n}$, is given by
	\begin{equation*}
%	\label{eq:adj} 
		\hat{f}=\sum_{i=0}^{n-1} a_i^{q^{n-i}}X^{q^{n-i}}.
	\end{equation*}
	
	The following result concerning the linear set defined by the adjoint of a $q$-polynomial was proved in \cite[Lemma 2.6]{bartoli_maximum_2018} and \cite[Lemma 2.1]{tang_equivalence_2022} in different ways.
	\begin{lemma}\label{le:LS_adjoint}
		Let $f$ be a $\sigma$-polynomial in $\mathcal{L}_{n,\sigma}$.Then for any $b\in\F_{q^n}$, there is
		\begin{equation*}
%		\label{eq:LS_adjoint}
		\#\left\{x\in\F_{q^n}^*  : \frac{f(x)}{x}=b \right\} = \#\left\{y\in\F_{q^n}^*  : \frac{\hat{f}(y)}{y}=b \right\}.
		\end{equation*}
		In particular, $L_f=L_{\hat{f}}$. 
	\end{lemma}
	
	Two linear sets $L_U$ and $L_{U'}$ of $\PG(1,q^n)$ are said to be $\PGamL$-\textit{equivalent} (or  \textit{projectively equivalent}) if there exists $\varphi\in \PGamL(2,q^n)$ such that $L_U^\varphi=L_{U'}$. 
	
	Let $L_U$ be an $\fq$-linear set of rank $n$ in $\Lambda$. Under the action of ${\rm P\Gamma L}(1,q^n)$ on $\Lambda$,  we may assume that $L_U$ does not contain the point $\langle(0,1)\rangle_{\fqn}$, so that $L_U$ can be written as
	\[ L_f=\{ \langle (x,f(x)) \rangle_{\fqn} \colon x \in \fqn^* \}, \]
	where $f$ is a $q$-polynomial over $\fqn$, that is $f=\sum_{i=0}^{n-1}a_iX^{q^i}\in \fqn[X]$.
	Sheekey in \cite{sheekey_new_2016} called $f$ \emph{scattered} if $L_f$ turns out to be a scattered $\fq$-linear set.
	
	For two $q$-polynomials $f$ and $g$ over $ \F_{q^n}$, we say $U_f=\{(x,f(x)): x\in \F_{q^n}^*\}$ and $U_g=\{(x,g(x)): x\in \F_{q^n}^*\}$ are $\GamL$-\textit{equivalent} if there exists $\varphi\in \GamL(2,q^n)$ such that $U_f^\varphi=U_{g}$. From the definitions, we get the following straightforward result.
	\begin{lemma}\label{lem:subeq_seteq}
		For $q$-polynomials $f$ and $g$ over $ \F_{q^n}$, if $U_f$ and $U_g$ are $\GamL$-\textit{equivalent}, then $L_f$ and $L_g$ are $\PGamL$-\textit{equivalent}.
	\end{lemma}

	However, the converse statement is not always true. For instance,  if $f(x)=x^q$ and $g(x)=x^{q^s}$ with $s\neq 1$ and $\gcd(s,n)=1$, then $U_f$ and $U_g$ are not $\GamL(2,q^n)$-equivalent, but obviously $L_f=\left\{\langle (1, x^{q-1})\rangle_{\F_{q^n}} :  x\in \F_{q^n}^* \right\}=L_g$. For more results on the equivalence problems, we refer to \cite{csajbok_classes_2018,csajbok_equivalence_2016}.
	
	In general, the equivalence problem between $U_f$ and $U_g$ is easier than the equivalence between $L_f$ and $L_g$. For instance, the equivalence between $U_f$ and $U_g$ with $f$, $g$ from family (b) has been completely determined; see the following lemma. However, the equivalence between $L_f$ and $L_g$ is still open and will be solved  in Section \ref{sec:4}.
	\begin{lemma}\label{lem:s_1,s_2}\cite[Proposition 5.1]{csajbok_classes_2018} 
		Define two $\F_q$-subspaces $U_{\delta,s_1}=\{(x,x^{q^{s_1}}+\delta x^{q^{s_1+m}}):x\in \mathbb{F}_{q^{2m}}\}$ and $U_{\theta, s_2}=\{(x,x^{q^{s_2}}+\theta x^{s_2+m}):x\in \mathbb{F}_{q^{2m}}\}$ of $\mathbb{F}_{q^{2m}}^2$ with $\delta $, $\theta \in\mathbb{F}_{q^{2m}}^*$. Assume that $N_{q^{2m}/{q^m}}(\delta ) \neq 1$, $N_{q^{2m}/{q^m}}(\theta ) \neq 1$, $1\leq s_1,s_2<m$ and $\gcd(m,s_1)=\gcd(m,s_2)=1$. The $\F_q$-subspaces $U_{\delta,s_1},U_{\theta, s_2}$ are 
		$\mathrm{\Gamma L}(2,q^{2m})$-equivalent if and only if either 
		\begin{equation*}
		s_1=s_2 \text{ and } N_{q^{2m}/{q^m}}(\delta ) = N_{q^{2m}/{q^m}}(\theta )^\sigma,
		\end{equation*}
		or 
		\begin{equation*}
		s_1+s_2=m  \text{ and } N_{q^{2m}/{q^m}}(\delta )  N_{q^{2m}/{q^m}}(\theta )^\sigma=1,
		\end{equation*}
		for some automorphism $\sigma\in \mathrm{Aut}(\F_{q^{m}})$.
	\end{lemma}
		
	Given two (multi-)subsets $S$ and $T\subseteq \F_{q^n}$, $S=T$ implies that 
	\[
	\sum_{x\in S} x^d=\sum_{x\in T} x^d,
	\]
	for any non-negative integer $d$. As a direct consequence, we can get the following necessary result for two identical linear sets.
	
	\begin{lemma} \label{lem3}\cite[Lemma 3.4]{B.csajbok_classes_2018}
		Let $f$ and $g$ be two linearized polynomials. If $L_f=L_g$, then for each non-negative integer $d$ the following holds  
		$$\sum_{x\in\mathbb{F}_{q^n}^*}\left(\frac{g(x)}{x}\right)^d=\sum_{x\in\mathbb{F}_{q^n}^*}\left(\frac{f(x)}{x}\right)^d.$$
	\end{lemma}

	The following result is well-known.
	\begin{lemma} \label{lem2}
		For any prime power $q$ and integer $d$ we have 
		\[
			\sum_{x\in\mathbb{F}_{q^n}^*}x^d=
			\begin{cases}
				-1,& q-1\mid d;\\
				0, &\text{otherwise}.
			\end{cases}
		\] 
	\end{lemma}
	
	By applying Lemmas \ref{lem3} and \ref{lem2}, one can prove the following criterion on the identity of two linear sets.
	\begin{lemma} \label{lem1}\cite[Lemma 3.6]{B.csajbok_classes_2018}
		Let $f(x)=\sum_{i=1}^n a_i x^{q^i}$ and $g(x)=\sum_{i=1}^n b_i x^{q^i}$ be two $q$-polynomials over $\mathbb{F}_{q^n}$,  such that $L_f=L_g$. Then
		\begin{equation}\label{eq:a_0_b_0}
			a_0=b_0;
		\end{equation}
		for $k=1, 2,\cdots, n-1$,
		\begin{equation}\label{eq:k_1_n-1}
			a_k a_{n-k}^{q^k}=b_k b_{n-k}^{q^k};
		\end{equation}
		for $k=2, 3,\cdots, n-1$,
		\begin{equation}\label{eq:k_2_n-1}
			a_1a_{k-1}^q a_{n-k}^{q^k}+a_k a_{n-1}^q a_{n-k+1}^{q^k}=b_1b_{k-1}^q b_{n-k}^{q^k}+b_k b_{n-1}^q b_{n-k+1}^{q^k}.
		\end{equation}
	\end{lemma}

	In general, the restrictions on the coefficients of $f$ and $g$ in Lemma \ref{lem1} are not enough to determine $g$ from $f$.  In Sections \ref{sec:3} and \ref{sec:4}, we will provide some extra restrictions  for $f$ and $g$ in (a) or (b).

	\section{Csajb\'ok-Marino-Montanucci-Zullo family}\label{sec:3}
	In this section, we consider the equivalence problem for the family of maximum scattered linear sets which is found for odd prime power $q\equiv \pm 1,0\pmod{5}$ by Csajb\'ok, Marino and Zullo in \cite{csajbok_newMSLS_2018} and later completely proved for any odd prime power $q$ by Marino, Montanucci and Zullo in \cite{marino_classes_2020}.	Recall that the linear sets in this family are of the shape
	\[
	\left\{\langle(1,,x^{q-1}+x^{q^3-1}+\theta x^{q^5-1}) \rangle_{\F_{q^{6}}} :x\in\F_{q^{6}}^* \right\},
	\]	
	where $\theta^2+\theta=1$.
	By definition, there are at most two elements in this family for given $q$ which is odd.  The following simple lemma tells us whether $\theta$ is in the prime field $\F_p$ of $\F_q$ or in $\F_{p^2}$.
	\begin{lemma}\label{lem:f_1,3,5}
		For positive integer $n$ and odd prime power $q$, suppose that $x_1^2+x_1=x_2^2+x_2=1$,  where $x_1,x_2 \in \F_{q^{n}}$. Then
		\begin{enumerate}[label=(\roman*)]
			\item $x_1,x_2\in \F_q$, if $q\equiv 0, \pm 1 \pmod{5}$;
			\item $x_1,x_2\in \F_{q^2}\setminus \F_q$,  if $q\equiv \pm 2 \pmod{5}$.
		\end{enumerate}
		Moreover,  $x_1^{2q-1}=x_2^{2q-1}$ if and only if $x_1=x_2$. In particular, when $q\equiv 0 \pmod 5$, $x_1=x_2=2$.
	\end{lemma}
	
	\begin{proof}
		%     Obviously,  $x_1,x_2$ are in $\F_{q^2}$, and  $x_1,x_2$ are in  $\mathbb{F}_{q}$ if and only if $5$ is a square  in  $\mathbb{F}_{q}$. 
		
		Let  $q=p^m$ such that  $p$ is prime. By the law of quadratic reciprocity, $5$ is a square in $\mathbb{F}_{p}$ if and only if $p=0,1$ or $4\pmod {5}$. 
		
		Suppose that $p=5$. The two roots $x_1$ and $x_2$ of the polynomial $X^2+X-1$ are identical and $x_1=x_2=2\in \F_p\subseteq \F_q$.  
		
		Suppose that $q\equiv \pm1 \pmod {5}$. When $p=1$ or $4\pmod {5}$,  $5$ is a square in $\mathbb{F}_{p}$ whence $5$ is a square in $\F_q$. When $p=2$ or $3\pmod {5}$ which means $q=p^{2k}$ for some $k\in \Z$, $5$ is also a square in $\mathbb{F}_{q}$.
		
		Suppose that $q\equiv \pm2 \pmod {5}$.  Then  $p=2$ or $3\pmod {5}$ and $q=p^{2k+1}$ for some $k\in \Z$. Hence  $5$ is not a square in $\mathbb{F}_{q}$.
		
		Next we show that the necessary and sufficient condition of $x_1^{2q-1}=x_2^{2q-1}$ is  $x_1=x_2$.
		The sufficiency is obvious. 
		
		Assume that  $x_1^{2q-1}=x_2^{2q-1}$. If $ q\equiv0,\pm 1 \pmod 5$, by (a) we get $x_1,x_2\in\F_q$. From $x_1=x_1^{2q-1}=x_2^{2q-1}=x_2$, we derive $ x_1=x_2$.  
		
		If $ q\equiv \pm 2 \pmod 5$, by (b) we get $x_1,x_2\in\F_{q^2} \setminus \mathbb{F}_{q}$. As $x_1,x_2$ are two (not necessarily distinct) roots of $x^2+x-1$, $x_1=x_2$ or $x_1=x_2^q$. Assume that $x_1=x_2^q$, by $x_1^{2q-1}=x_2^{2q-1}$ and $x_2^{q+1}=-1$ we get $x_2^6=-1$. As $x_2^2+x_2=1$,   $$(1-x_2)^3=-x_2^3+3x_2^2-3x_2+1=x_2(4x_2-4)+1=5-8x_2=-1,$$ which implies $$x_2=3/4.$$ However, this contradicts $x_2^2+x_2=1$ under the assumption that $p \neq 2,5$.  Therefore $x_1$ must be identical to $x_2$.
	\end{proof}
	
%	Define $U_f=\left\{(x,f(x)):x\in\F_{q^6}\right\}$ and $U_f^q=\left\{(x^q,f(x)^q):x\in\F_{q^6}\right\}.$
	Before looking at maximum scattered linear sets in the Csajb\'ok-Marino-Montanucci-Zullo family, we first consider the equivalence between the $6$-dimensional $\F_{q}$-subspaces $U_f$ associated with different $L_f$ from this family. Depending on the value of $q$ mod $5$, our result are divided into two lemmas.
	\begin{lemma}\label{le:U_equivalence_0}
		For odd prime power $q$ with $ q\equiv \pm 2 \pmod 5$, define functions $f(x)=x^q+x^{q^3}+\theta x^{q^5}$, $g(x)=x^q+x^{q^3}+\delta x^{q^5}$ on $\F_{q^6}$ with $\theta$, $\delta \in \F_{q^6}$ satisfying $\theta ^2+ \theta =1$ and $\delta^2+\delta=1$. Then $\F_q$-subspaces $U_f$ and $U_g$ are $\GamL$-equivalent.
	\end{lemma}
	\begin{proof}
		As $\theta,\delta$ are the roots of $x^2+x-1$ which is irreducible over $\F_{q}$, $\theta=\delta$ or $\theta=\delta^q$.

		If $\theta=\delta$, then $U_f=U_g$;	if $\theta=\delta^q$, then
		\[
		\left\{(x^q,(x^q+x^{q^3}+\delta x^{q^5})^q):x\in\F_{q^6}\right\}=\left\{(x,x^q+x^{q^3}+\delta^q x^{q^5}):x\in\F_{q^6}\right\}=U_f. 
		\]
		Therefore, $U_f$ and $U_g$ are $\GamL$-equivalent.
	\end{proof}
	
	\begin{lemma}\label{le:U_equivalence_1}
		For odd prime power $q$ with $ q\equiv 0,\pm 1 \pmod 5$, define functions $f(x)=x^q+x^{q^3}+\theta x^{q^5}$, $g(x)=x^q+x^{q^3}+\delta x^{q^5}$ on $\F_{q^6}$ with $\theta$, $\delta \in \F_{q^6}$ satisfying $\theta ^2+ \theta =1$ and $\delta^2+\delta=1$. If $\F_q$-subspace $U_f$ and $U_g$ are $\GamL$-equivalent,
		then  $\theta=\delta.$
	\end{lemma}
	\begin{proof}
		Assume that $\F_q$-subspace $U_f$ and $U_g$ are $\GamL$-equivalent.
		By Lemma \ref{lem:f_1,3,5}, $\theta,\delta\in\F_{q}$. We only have to consider the existence of invertible matrix $M=\begin{pmatrix}
		a & b   \\ 
		c& d  
		\end{pmatrix}$ over $\F_{q^6}$ such that  for each $x\in\F_{q^6}$ there exists $y\in\F_{q^6}$ satisfying 
		\begin{equation*}
%		\label{eq:U_main_equivalence_1,3,5}
		\begin{pmatrix}
		a & b   \\ 
		c& d  
		\end{pmatrix}
		\begin{pmatrix}
		x  \\ 
		f(x) 
		\end{pmatrix}=\begin{pmatrix}
		y  \\ 
		g(y)  
		\end{pmatrix}.
		\end{equation*} 
		
		This is equivalent to 
		\begin{equation}\label{eq:u_1}
		c X+d f(X)\equiv g\left(a X+bf(X)\right) \pmod{X^{q^6}-X}.
		\end{equation}
		Then the right-hand-side of \eqref{eq:u_1} is 
		\begin{align*}
		&\left(b^q\theta+b^{q^3}+b^{q^5}\delta\right)X+a^q X^q+\left(b^q+b^{q^3}\theta+b^{q^5}\delta\right)X^{q^2}\\
		&+a^{q^3}X^{q^3}+\left(b^q+b^{q^3}+b^{q^5}\theta \delta\right)X^{q^4}+a^{q^5}\delta X^{q^5}.
		\end{align*}
		By the coefficients of $X^q,X^{q^3}$ and $X^{q^5}$ in \eqref{eq:u_1}, we get 
		\begin{equation}\label{eq:u_1,3,5}
		\begin{cases}
		d=a^q=a^{q^3},\\
		d\theta=a^{q^5}\delta.
		\end{cases}
		\end{equation}
		
		Depending on the value of $b$, we consider two cases.
		\medskip
		
		\noindent\textbf{Case 1:} $b=0$. By the coefficients of $X$ in \eqref{eq:u_1}, we get $c=0$ and $ad\neq 0$. Then  \eqref{eq:u_1,3,5} is equivalent to  $a,d\in\F_{q^2}$ and 
		\begin{equation*}	
		\begin{cases}
		d=a^q,\\
		\theta=\delta.
		\end{cases}
		\end{equation*}
		
		\medskip
		\noindent\textbf{Case 2:} $b\neq 0$. By the coefficients of $X^{q^2}$ and $X^{q^4}$ in \eqref{eq:u_1}, we get
		\begin{equation}\label{eq:u_b_1_2}
		\begin{cases}
		b^q+b^{q^3}\theta+b^{q^5}\delta=0,\\
		b^q+b^{q^3}+b^{q^5}\theta \delta=0.
		\end{cases}
		\end{equation}
		
		By  \eqref{eq:u_b_1_2}, we obtain 
		$$\left(b^{q^3}-b^{q^5}\delta\right)\left(1-\theta \right)=0,$$
		which means at least one of $b^{q^3}-b^{q^5}\delta$ and $1-\theta$ equals $0$. Clearly  $1-\theta=0$ contradicts the condition $\theta^2+\theta=1$. If $b^{q^3}-b^{q^5}\delta=0$,   then $\delta=b^{q^3-q^5},$
		which implies 
		\[\delta^{q^4+q^2+1}=N_{{q^6}/{q^2}}\left(b^{q^3(1-q^2)}\right)=1.\]
		
		However, by $\delta ^2+ \delta =1$ and $\delta\in \F_{q^2}$, we get
		$$\delta^{q^4+q^2+1}=\delta^3=(1-\delta)\delta=2\delta-1.$$
		Thus the above two equations on $\delta^{q^4+q^2+1}$ tells us $\delta=1$ which contradicts $\delta ^2+ \delta =1$. Therefore, there is no  $b\in \F_{q^6}^*$  such  that $\delta=b^{q^3-q^5}$.
		
		To summarize, we have shown that $\theta$ must be identical to $\delta$ and the matrix $M=\begin{pmatrix}
		a& 0 \\ 
		0& a^q
		\end{pmatrix} $.
	\end{proof}
	
%	
%	\myue{Define Marino family $\cF(q)=\{L_f:f(x)=x^q+x^{q^3}+\theta x^{q^5},\theta^2+\theta=1\,\text{ and }x,\theta \in \F_{q^6}\}.$}
	
	The next result completely solve the equivalence of the members in the Csajb\'ok-Marino-Montanucci-Zullo family. In particular, compared with Lemma \ref{le:U_equivalence_1}, we can see that the equivalence of $U_f$ and $U_g$ are not the same as the equivalence of $L_f$ and $L_g$ in this family.
	\begin{theorem}\label{th:LP_equivalence}
		Let  $q$ be an odd  prime power. All elements of the Csajb\'ok-Marino-Montanucci-Zullo family  are $\PGamL$-equivalent. 
	\end{theorem}
	\begin{proof}
		By Lemma \ref{lem:f_1,3,5}, when $q\equiv 0 \pmod 5$  there are only one element in this family and when $q\equiv \pm 1,\pm 2 \pmod 5$ there are only two elements. 

		For $q\not \equiv 0 \pmod 5$, let $f_i(x)=x^q+x^{q^3}+\theta_i x^{q^5}$ with $i=1,2$ stand for the two members in the  family.
		
	    When $q\equiv\pm 2 \pmod 5$, Lemma \ref{le:U_equivalence_0} tells us that $U_{f_1}$ and $U_{f_2}$ are $\GamL$-equivalent, which means that  $L_{f_1}$ and $L_{f_2}$ are equivalent by Lemma \ref{lem:subeq_seteq}.
		
		When $q\equiv \pm 1 \pmod 5$, by Lemma \ref{lem:f_1,3,5} we get $\theta_1,\theta_2 \in \F_{q}$. 
		Let $h(x)=(-\theta_1 -1)x^q+x^{q^3}+x^{q^5}$. Then it is straightforward to verify $h(f_1(x))=x$, i.e.\ $h(x)$ is the inverse map of $f_1(x)$. Hence
		\begin{equation*}
		\left\{ \frac{x}{f_1(x)}: x\in \F_{q^6}^*  \right\}=\left\{\frac{(-\theta_1 -1)x^q+x^{q^3}+x^{q^5}}{x} : x\in \F_{q^6}^* \right  \}=\left\{\frac{h(x)}{x} : x\in \F_{q^6}^* \right  \},
		\end{equation*}
		which means $L_{f_1}$ is $\PGamL$-equivalent to $L_{h}$.
		
		By a simple calculation, $$\hat{h}(x)= x^q +x^{q^3} + \left(-\theta_1-1\right)^{q^5}x^{q^5}= x^q +x^{q^3} + \theta_2 x^{q^5}=f_2(x).$$ By Lemma \ref{le:LS_adjoint}, $L_{h}=L_{\hat{h}}=L_{f_2}$. Therefore $L_{f_1}$ is $\PGamL$-equivalent to $L_{f_2}$.
	\end{proof}

	By Theorem \ref{th:LP_equivalence}, the Csajb\'ok-Marino-Montanucci-Zullo family contains only one member up to equivalence. Next we determine the automorphism group of this element.
	\begin{theorem}
%		\label{th:Aut}
		Let  $q$ be an odd  prime power. Let $f(x)=x^q+x^{q^3}+\theta x^{q^5}$ over $\F_{q^6}$ with $\theta^2+\theta=1$. The automorphism  group of $L_f$ is
				\[\Aut(L_f)=\begin{cases}
					\mathcal{D}, & q\equiv\pm 1 \pmod 5;\\
					\mathcal{D}\cup\mathcal{C}, & q\equiv0,\pm 2 \pmod 5,
				\end{cases}
				\]
	where
	$$\mathcal{D}:=\left\{M\tau: M=	\begin{pmatrix} 
		1 & 0   \\ 
		0 & d  
	\end{pmatrix}\in\mathrm{PGL}(2,q^6), \tau\in \mathrm {Aut}(\F_{q^6}),\theta^{\tau-1}=1,d^{q+1}=1 \right\},$$
	
	and
	$$\mathcal{C}:=\left\{M\tau: M=	\begin{pmatrix} 
		0 & 1   \\ 
		c & 0  
	\end{pmatrix}\in\mathrm{PGL}(2,q^6), \tau\in \mathrm {Aut}(\F_{q^6}), \theta^{\tau+1}=-1,  c^{q+1}=1 \right\}.$$
	\end{theorem}
	
	\begin{proof}

		Let $M=
		\begin{pmatrix}
			a & b   \\ 
			c& d  
		\end{pmatrix}$ be an element in $\mathrm{PGL}(2,q^6)$ and  $\tau \in \mathrm {Aut}(\F_{q^6}/\F_q)$. Assume that $M\tau\in\mathrm{P\Gamma L}(2,q^6)$ is an automorphism of $L_f$,  that is
		\begin{equation}\label{eq:LP_main_equivalence_1,3,5}
		\left\{ \frac{f(x)}{x} : x\in \F_{q^6}^*  \right\} = \left\{\frac{cx+d\bar{f}(x)}{ax+b\bar{f}(x)}:x\in \mathbb{F}_{q^6}^*\right\},
		\end{equation} 
		where $\bar{f}(x)=x^{q}+x^{q^3}+\theta^\tau x^{q^5}$. 
		
	%	By Lemma \ref{lem:f_1,3,5}, we get $\theta\in \F_{q^2}.$
		
		Next we determine $M$ and $\tau$. Depending on the value of $b$, we separate the rest part into two cases.
		
		\medskip
		
		\noindent\textbf{Case 1:} $b=0$. As $M$ is viewed as an element in $\PGL(2,q^6)$, we can assume that $a=1$ and \eqref{eq:LP_main_equivalence_1,3,5} becomes 
		\begin{equation}\label{eq:add_case1}
		\left\{ x^{q-1}+x^{q^3-1}+\theta x^{q^5-1} : x\in \F_{q^6}^*  \right\} = \left\{ c+d\frac{\bar{f}(x)}{x}: x\in \F_{q^6}^*  \right\}.
		\end{equation}
		
		By \eqref{eq:a_0_b_0}, \eqref{eq:k_1_n-1} for $k=1$ and $3$ in Lemma \ref{lem1}, respectively,  we derive $c=0$,
		\begin{equation}\label{eq:f_1,3,5_d1}
		    d^{q+1}=\left(\frac{\theta }{\theta^\tau}\right) ^q,
		\end{equation}
		and
		\begin{equation}\label{eq:f_1,3,5_d2}
		    d^{q^3+1}=1.
		\end{equation}
		By \eqref{eq:f_1,3,5_d1} and \eqref{eq:f_1,3,5_d2}, we can get 
		\begin{equation}\label{eq:f_1,3,5_b_0}
		\theta^{\tau \left(q^3-q^2+q\right)}=\theta^{q^3-q^2+q}.
		\end{equation}
		By Lemma \ref{lem:f_1,3,5}, \eqref{eq:f_1,3,5_b_0} is equivalent to 
		
		\begin{equation}\label{eq:f_1,3,5_1}
		\theta^\tau=\theta.
		\end{equation}
		By \eqref{eq:f_1,3,5_d1}, \eqref{eq:f_1,3,5_d2} and \eqref{eq:f_1,3,5_1}, 
		\begin{equation} 
		\label{eq:f_1,3,5_d3}
		d^{q+1}=1.
		\end{equation}
		%%By Lemma \ref{lem:f_1,3,5} and \eqref{eq:f_1,3,5_1}, we get
		%%\begin{equation}\label{eq:f_1,3,5_2}
		%\begin{cases}
		%	\theta=\delta, & q\equiv0,\pm 1 \pmod 5; \\
		%		\theta=\delta \text{ or } \theta=-\frac{1}{\delta}, &q\equiv \pm 2 \pmod 5.
		%	\end{cases}
		%%\end{equation}

		Then  \eqref{eq:f_1,3,5_1} and \eqref{eq:f_1,3,5_d3} are the necessary conditions.
		Next, we  show that \eqref{eq:f_1,3,5_1} and \eqref{eq:f_1,3,5_d3} are also sufficient. By \eqref{eq:f_1,3,5_d3}, it is easy to see that there exists $\alpha \in \F_{q^2}^*$ such that $d=\alpha^{q-1}$. Replacing $x$ by $\alpha x$ in \eqref{eq:add_case1}, we obtain
		\begin{equation}
		\label{eq:alpha_1}
		\left\{ (\alpha x)^{q-1}+(\alpha x)^{q^3-1}+\theta (\alpha x)^{q^5-1} : x\in \F_{q^6}^*  \right\} = \left\{ \alpha^{q-1}\left(x^{q-1}+x^{q^3-1}+\theta x^{q^5-1}\right): x\in \F_{q^6}^*  \right\}.
		\end{equation}
		By \eqref{eq:f_1,3,5_1},  \eqref{eq:alpha_1} is equivalent to 
		$$	\left\{  x^{q-1}+x^{q^3-1}+\theta x^{q^5-1} : x\in \F_{q^6}^*  \right\} =\left\{ d\left(x^{q-1}+x^{q^3-1}+\theta^\tau x^{q^5-1}\right): x\in \F_{q^6}^*  \right\}. $$
		
		Therefore, $L_f$ is mapped to itself under $M\tau$ where
		$M=
		\begin{pmatrix}
			1 & 0   \\ 
			0& d  
		\end{pmatrix}$ and $\tau$ satisfy conditions \eqref{eq:f_1,3,5_1} and \eqref{eq:f_1,3,5_d3}.

		\medskip
		\noindent\textbf{Case 2:} $b\neq 0$.  Without loss of generality, we assume that $b=1$. 
		First we concentrate on the proof of the following claim:
		
		\medskip
		\textbf{Claim:} $a=0$.
		
		Assume, by way of contradiction, that $a\neq 0$. Now
		\[\frac{cx+d\bar{f}(x)}{ax+b\bar{f}(x)}=d+\frac{(c-ad)x}{ax+\bar{f}(x)},\]
		where $\bar{f}(x)=x^q+x^{q^3}+\theta^\tau x^{q^5}.$
		
		As $M$ is nonsingular, $c-ad\neq 0$. Furthermore, by definition of maximum scattered linear sets in $\PG(1,q^{6})$, the map $x\mapsto ax+\bar{f}(x)$ must be invertible; otherwise the point $\langle (0,1)\rangle_{\F_{q^{6}}}$ would be in $L_f$ which is impossible.
				
		Let $ h(x)=\sum_{i=0}^{5} r_i x^{q^i} $ denote the inverse map of $x\mapsto \frac{ ax+\bar{f}(x)}{c-ad}$. Setting $y=\frac{ ax+\bar{f}(x)}{c-ad}$, we get
		\[
		d+\frac{(c-ad)x}{ax+\bar{f}(x)}=d+\frac{h(y)}{y}.
		\]
		By comparing the coefficients of $ah(Y)+\bar{f}(h(Y))\equiv (c-ad)Y \pmod{Y^{q^6}-Y}$, we obtain
%		\begin{equation}\label{eq:d=0_1}
%		a r_0 +r_{5}^{q}+r_{3}^{q^3}+\theta^\tau r_1^{q^{5}} = c-ad
%		\end{equation} 
%		and
		\begin{equation}\label{eq:d=0_2}
		a r_i +r_{i-1}^{q}+r_{i+3}^{q^3}+\theta^\tau r_{i+1}^{q^5} = 0,
		\end{equation} 	
		for $i=1,2,3,4,5$.
		
		Now \eqref{eq:LP_main_equivalence_1,3,5} becomes 
		\begin{equation*}
%		\label{eq:LP_1,3,5_b=1}
		\left\{ x^{q-1}+x^{q^3-1}+\theta x^{q^5-1} : x\in \F_{q^6}^*  \right\} = \left\{ d+\sum_{i=0}^{5}r_i y^{q^i-1}: y\in \F_{q^5}^*  \right\}.
		\end{equation*}
		By \eqref{eq:k_1_n-1} for $k=1$, 	
		\begin{equation*}
		r_1r_5^q=\theta^q,
		\end{equation*}
		which means  $r_1,r_5 \neq 0$ .
		By \eqref{eq:k_1_n-1} for $k=2$, 	
		\begin{equation*}
		r_2r_4^{q^2}=0,
		\end{equation*}
		which means at least one of $r_{2}$ and $r_{4}$ equals $0$. Without loss of generality, we assume that $r_2=0$. 
		By \eqref{eq:k_2_n-1} for $k=2$, 	
		\begin{equation*}
		r_1^{q+1}r_4^{q^2}+r_2r_5^{q^2+q}=0
		\end{equation*}
		which means  $r_1^{q+1}r_4^{q^2}= 0$, then $r_4=0$.
		
		By \eqref{eq:d=0_2} for $i=1,3,5$ and the assumption that $a\neq 0$, we get
		\begin{equation}\label{eq:i_1,3,5}
		\begin{cases}
		r_1=-\frac{r_0^q}{a},\\
		r_3=-\frac{r_0^{q^3}}{a},\\
		r_5=-\frac{\theta^\tau r_0^{q^5}}{a}.
		\end{cases}
		\end{equation}
		By \eqref{eq:d=0_2} for $i=2$, we get
		\begin{equation*}
		r_1^q+r_5^{q^3}+\theta^\tau r_3^{q^5}=0,
		\end{equation*}
		and since $\theta^\tau\in\F_{q^2}$, which is equivalent to	
		\begin{equation}\label{eq:i_2}
		r_1^{q^5}+r_5^{q}+\theta^\tau r_3^{q^3}=0,
		\end{equation}
		Similarly by \eqref{eq:d=0_2} for $i=4$, we get
		%\begin{equation*}
		%r_3^q+r_1^{q^3}+\theta^\tau r_5^{q^5}=0,
		%\end{equation*}
		%which is equivalent to	
		\begin{equation}\label{eq:i_4}
		r_1^{q^5}+r_3^{q^3}+\theta^\tau r_5^{q}=0.
		\end{equation}
		By \eqref{eq:i_2} and \eqref{eq:i_4}, we get $\left(r_3^{q^3}-r_5^{q}\right)(\theta^\tau-1)=0$, which implies
		\begin{equation*}
	    	r_3^{q^3}=r_5^{q}.
		\end{equation*}
		Plugging \eqref{eq:i_1,3,5} into the above equation, we have $\theta^\tau=a^{q-q^3}$
		which implies
		\begin{equation}
		\label{eq:theta_1}
		\theta^{3\tau }=N_{{q^6}/{q^2}}(\theta^\tau)=N_{{q^6}/{q^2}}\left(a^{q(1-q^2)}\right)=1.
		\end{equation}
		By $\theta ^2+ \theta =1$ and $\theta\in \F_{q^2}$, we get
		$$\theta^{3\tau }=(1-\theta^\tau)\theta^\tau=\theta^\tau+\theta^\tau-1=2\theta^\tau-1,$$
		together with \eqref {eq:theta_1}, we get $\theta=1$, which is a contradiction to the assumption that $\theta^2+\theta=1$. 
		
		Therefore, we have proved the claim that $a$ must be $0$.
		\medskip
		
		As $$
		\begin{pmatrix} 
		0 & 1   \\ 
		c & d  
		\end{pmatrix}^{-1}=	
		\begin{pmatrix} 
		-d/c & 1/c   \\
		1 & 0 
		\end{pmatrix},$$
		by applying \textbf{Claim} on  the inverse map of $M$ with $a=0$, we get $d = 0$. Thus
		$$M= 
		\begin{pmatrix} 
		0 & 1   \\ 
		c & 0  
		\end{pmatrix}.$$
		
		Let $h(x)=-(\theta^{\tau q}+1)x^q+x^{q^3}+x^{q^5}.$ 
		It can be readily verified that $h(\bar{f}(x))=x$, i.e.\ $h(x)$ is the inverse map of $\bar{f}(x)$. Consequently, \eqref{eq:LP_main_equivalence_1,3,5} equals 
		\begin{equation*}
		\left\{\frac{ch(x)}{x}:x\in \mathbb{F}_{q^6}^*\right\}=\left\{ \frac{f(x)}{x}:x\in \mathbb{F}_{q^6}^*\right\},
		\end{equation*}
		By \eqref{eq:k_1_n-1} for $k=1$ in Lemma \ref{lem1}, we get \begin{equation} 
		\label{eq:f_1,3,5_c1}
		c^{q+1}=\frac{-\theta^{q}}{\left(\theta^{\tau q}+1\right)},
		\end{equation}
		and by \eqref{eq:k_1_n-1} for $k=3$ we obtain
		\begin{equation} 
		\label{eq:f_1,3,5_c2}
		c^{q^3+1}=1.
		\end{equation}
		By \eqref{eq:f_1,3,5_c1} and \eqref{eq:f_1,3,5_c2}, we get 
		\begin{equation}\label{eq:f_1,3,5_b_1}
		\theta^{2q-1}=\left(-\theta^\tau-1\right)^{2q-1}.
		\end{equation}
		As $\theta$ and $-\theta^\tau-1$ are both roots of $X^2+X-1$, by Lemma \ref{lem:f_1,3,5}, \eqref{eq:f_1,3,5_b_1} is equivalent to $-\theta=\theta^{\tau}+1$. It follows that
		\begin{equation}\label{eq:f_1,3,5_3}
		\theta^{\tau+1}= -1.
		\end{equation}
		
		By \eqref{eq:f_1,3,5_c1} and $\theta^\tau+1=-\theta$,  \begin{equation}
		\label{eq:f_1,3,5_5}
		c^{q+1}=1.
		\end{equation} 
		Then  \eqref{eq:f_1,3,5_3} and \eqref{eq:f_1,3,5_5} are the necessary conditions. 
		
		Next we show that \eqref{eq:f_1,3,5_3} and \eqref{eq:f_1,3,5_5} are sufficient. 
		Suppose \eqref{eq:f_1,3,5_3} and \eqref{eq:f_1,3,5_5} hold.
		Let $h(x)=-(\theta^{\tau q}+1)x^q+x^{q^3}+x^{q^5}.$ 
		Then $h( \bar{f}(x))=x$ which means
		\begin{equation}\label{eq:f_1,3,5_b_1_1}	\left\{ \frac{cx}{\bar{f}(x)}:x\in \mathbb{F}_{q^6}^*\right\}=\left\{\frac{ch(x)}{x}:x\in \mathbb{F}_{q^6}^*\right\}.
		\end{equation}
		By a simple calculation, the adjoint $\hat{h}(x)$ of $h(x)$ is
		$$\hat{h}(x)= x^q +x^{q^3} + \left(-\theta^{\tau q}-1\right)^{q^5}x^{q^5}= x^q +x^{q^3} + \theta x^{q^5}=f(x).$$
		By Lemma \ref{le:LS_adjoint}, $L_{h}=L_{\hat{h}}$, which means
		\begin{equation}
		\label{eq:f_1,3,5_b_1_2}\left\{\frac{ch(x)}{x}:x\in \mathbb{F}_{q^6}^*\right\}=	\left\{\frac{c\hat{h}(x)}{x}:x\in \mathbb{F}_{q^6}^*\right\}=	\left\{\frac{ cf(x)}{x}:x\in \mathbb{F}_{q^6}^*\right\}.
		\end{equation}
		
		The \textbf{Case 1} of the proof tells us that if $c$ satisfies \eqref{eq:f_1,3,5_5}, then
		\begin{equation}
		\label{eq:f_1,3,5_b_1_3}
		\left\{\frac{ cf(x)}{x}:x\in \mathbb{F}_{q^6}^*\right\}=\left\{\frac{f(x)}{x}:x\in \mathbb{F}_{q^6}^*\right\}.
		\end{equation}
		By \eqref{eq:f_1,3,5_b_1_1}, \eqref{eq:f_1,3,5_b_1_2} and \eqref{eq:f_1,3,5_b_1_3}, we get that \eqref{eq:f_1,3,5_3} and \eqref{eq:f_1,3,5_5} are  the sufficient conditions. 
		
		Now the automorphism  group of $L_f$ is
			\[\Aut(L_f)=\mathcal{D}\cup\mathcal{C},	\]
		where
		$$\mathcal{D}:=\left\{M\tau: M=	\begin{pmatrix} 
			1 & 0   \\ 
			0 & d  
		\end{pmatrix} ,\theta^\tau=\theta,d^{q+1}=1 \right\},$$
		
and
		$$\mathcal{C}:=\left\{M\tau: M=	\begin{pmatrix} 
			0 & 1   \\ 
			c & 0  
		\end{pmatrix}, \theta^{\tau+1}=-1,  c^{q+1}=1 \right\}.$$
	By Lemma \ref{lem:f_1,3,5} and $\theta^2+\theta=1$, we get that the equation $\theta^{\tau+1}=-1$ holds if and only if  one of the following collections of conditions is satisfied: 	
	\begin{itemize}
		\item
		$q\equiv 0  \pmod 5$, $\theta^\tau=\theta$
		\item
		$q\equiv \pm 2 \pmod 5$, $\theta^\tau=\theta^q$.
	\end{itemize} 
	In particular, when $q\equiv \pm 1 \pmod{5}$,  $\theta^{\tau+1}=-1$  can never be satisfied and $\mathcal{C}=\varnothing$.
	\end{proof}

	\section{Csjab\'ok-Marino-Polverino-Zanella Families}\label{sec:4}
	
	In this section, we consider the equivalence problem for the two families $\CFA$ and $\CFB$ of maximum scattered linear sets constructed by Csjab\'ok, Marino, Polverino and Zanella in \cite{csajbok_classes_2018}.
	
	Recall that $\CF_{2m}$  contains maximum linear sets in $\PG(1,q^{2m})$ of the following shape
	\[
	\left\{\langle(1,x^{q^s-1}+\delta x^{q^{s+m}-1}) \rangle_{\F_{q^{2m}}} :x\in\F_{q^{2m}}^* \right\}
	\]
	for	$m=3,4$, where $\gcd(s,m)=1$ and $\delta$, $q$ satisfy certain conditions.
	
	By Lemma \ref{lem:subeq_seteq}, the following result tells us that we only have to consider the members in $\CF_{2m}$ with $s=1$.

	\begin{lemma}\cite{csajbok_classes_2018}
%		\label{lem:U_b_m}
		Every subspace $U_{b',s}=\{(x,x^{q^s}+b'x^{q^{s+m}}):x\in\F_{q^{2m}}\}$ with $\gcd(s,m)=1$, is $\GL(2,q^{2m})$-equivalent to a subspace $$U_{b}=\{(x,x^q+bx^{q^{1+m}}):x\in\F_{q^{2m}}\}, $$for some $b\in \F_{q^{2m}}$.
	\end{lemma}

	To handle the equivalence problem for $\CFA$ and $\CFB$, we need to use Lemma \ref{lem1}. However, it is not enough to provide us a complete answer, and we need the following several extra lemmas.
	\begin{lemma}\label{D_6}
		Let $f(x)=a_1x^{q}+a_4x^{q^4}$ and   $g(x)=b_1x^{q}+b_4 x^{q^4}$  be two q-polynomials over $\mathbb{F}_{q^6}$. If $L_f=L_g$, then 
		\begin{equation}	\label{eq:(2,0)_1,1}
		a_4^{q^4+q^2+1}=b_4^{q^4+q^2+1}.
		\end{equation}
	\end{lemma}
	\begin{proof}
		Let $D=(q^2-q+1)(q^2+q+1)=q^4+q^2+1$. Then 
		$$\sum_{x\in\mathbb{F}_{q^6}^*}\left(\frac{f(x)}{x}\right)^D=\sum_{u_1,u_2,u_3\in\{1,4\}}a_{u_1}a_{u_2}^{q^2}a_{u_3}^{q^4} \sum_{x\in\mathbb{F}_{q^6}^*}x^U,$$
		with $U={q^{u_1}+ q^{u_2+2}+q^{u_3+4}-(1+q^2+q^4) }$. 
		By Lemma \ref{lem3}, we obtain $\sum_{x\in\mathbb{F}_{q^6}^*}x^k=-1$ if and only if $q^6-1|k$. When $\sum_{x\in\mathbb{F}_{q^6}^*}x^U=-1$ ,
		\begin{equation*}
			q^{u_1}+ q^{u_2+2}+q^{u_3+4}\equiv 1+q^2+q^4 \pmod {q^6-1},
		\end{equation*}
		with $u_1,u_2,u_3\in\{1,4\}.$ A simple calculation shows that $u_1=u_2=u_3=4$.
		
		By $L_f=L_g$  and Lemma \ref{lem2}, we get
		\begin{equation*}
			\sum_{x\in\mathbb{F}_{q^6}^*}\left(\frac{f(x)}{x}\right)^D=\sum_{x\in\mathbb{F}_{q^6}^*}\left(\frac{g(x)}{x}\right)^D.
		\end{equation*}
		which means,  
		\[
			a_4a_4^{q^2}a_4^{q^4}=b_4b_4^{q^2}b_4^{q^4}. \qedhere
		\]
	\end{proof}

	\begin{lemma}\label{D_8}
		Let $f(x)=a_1x^{q}+a_5x^{q^5}$ and   $g(x)=b_1x^{q}+b_5 x^{q^5}$  be two q-polynomials over $\mathbb{F}_{q^8}$. If $L_f=L_g$, then we have 
		\begin{equation}	\label{eq:(3,2,1,0)_1,1}
			a_1^{q^2+q+1}a_5^{q^3}=b_1^{q^2+q+1}b_5^{q^3},
		\end{equation}
		and
	\begin{equation}	\label{eq:(3,2,0)(3,2,1,0)_1,1}
		a_1a_5^{q^6+q^3+q}=b_1b_5^{q^6+q^3+q}.
	\end{equation}
	\end{lemma}
	\begin{proof}
		Let $D_1=q^3+q^2+q+1$, we get 
		$$\sum_{x\in\mathbb{F}_{q^8}^*}\left(\frac{f(x)}{x}\right)^{D_1}=\sum_{u_1,u_2,u_3,u_4\in\{1,5\}}a_{u_1}a_{u_2}^qa_{u_3}^{q^2}a_{u_4}^{q^3} \sum_{x\in\mathbb{F}_{q^8}^*}x^U,$$ with $U=q^{u_1}+ q^{u_2+1}+q^{u_3+2}+q^{u_4+3} -(1+q+q^2+q^3).$ 
		By Lemma \ref{lem3}, we obtain $\sum_{x\in\mathbb{F}_{q^8}^*}x^k=-1$ if and only if $(q^8-1)|k$. When $\sum_{x\in\mathbb{F}_{q^8}^*}x^U=-1$, 	
		\begin{equation*}
			q^{u_1}+ q^{u_2+1}+q^{u_3+2}+q^{u_4+3} \equiv 1+q+q^2+q^3 \pmod {q^8-1},
		\end{equation*}
		with $u_1,u_2,u_3,u_4\in\{1,5\}$. A direct computation shows that $u_1=u_2=u_3=1,u_4=5$.

	Let $D_2=(q^3-q^2+1)(q^3+q^2+q+1)=q^6+q^3+q+1$, we get 
	$$\sum_{x\in\mathbb{F}_{q^8}^*}\left(\frac{f(x)}{x}\right)^{D_2}=\sum_{i_1,i_2,i_3,i_4\in\{1,5\}}a_{i_1}a_{i_2}^qa_{i_3}^{q^3}a_{i_4}^{q^6} \sum_{x\in\mathbb{F}_{q^8}^*}x^I,$$ with $I=q^{i_1}+ q^{i_2+1}+q^{i_3+3}+q^{i_4+6} -(1+q+q^3+q^6)$.
   When $\sum_{x\in\mathbb{F}_{q^8}^*}x^I=-1$, 	
   \begin{equation*}
%		\label{eq:I_8}
		q^{i_1}+ q^{i_2+1}+q^{i_3+3}+q^{i_4+6} \equiv 1+q+q^3+q^6 \pmod {q^8-1},
	\end{equation*}
	with $i_1,i_2,i_3,i_4\in\{1,5\}$. Hence we must have $i_1=1,i_2=i_3=i_4=5$.
	
	By $L_f=L_g$  and Lemma \ref{lem2}, we get
	\begin{equation*}
%		\label{eq:f_D_k=g_D_k}
		\sum_{x\in\mathbb{F}_{q^6}^*}\left(\frac{f(x)}{x}\right)^{D_k}=\sum_{x\in\mathbb{F}_{q^6}^*}\left(\frac{g(x)}{x}\right)^{D_k},
	\end{equation*}
	with $k=1,2$,
	which means
			$$a_1a_1^qa_1^{q^2}a_5^{q^3}=b_1b_1^qb_1^{q^2}b_5^{q^3},$$
	and
	\begin{equation*}
		a_1a_5^{q}a_5^{q^3}a_5^{q^6}=b_1b_5^{q}b_5^{q^3}b_5^{q^6}.\qedhere
	\end{equation*}	
\end{proof}

	\begin{lemma}\label{d=0}
		Let $f(x)=x^{q^{s_1}}+\delta x^{q^{s_1+m}}$, $g(x)=x^{q^{s_2}}+\theta x^{q^{s_2+m}}$ with $\delta, \theta \in \F_{q^{2m}}$, $1\leq s_1,s_2<m$ and $\gcd(m,s_1)=\gcd(m,s_2)=1$. 
%		Assume that $L_f$ and $L_g$ are maximum scattered linear sets. 
		Let $M=
		\begin{pmatrix}
			a & 1   \\ 
			c& d  
		\end{pmatrix}$ be an invertible matrix  over $\F_{q^{2m}}$. If 
		\begin{equation}\label{eq:d=0}
			\left\{ x^{q^{s_1}-1}+\delta x^{q^{s_1+m}-1} : x\in \F_{q^{2m}}^*  \right\} = \left\{ \frac{cx+dg(x)}{ax+g(x)}: x\in \F_{q^{2m}}^*  \right\}, 
		\end{equation}
		then $a=d=0$.
	\end{lemma}
	\begin{proof}
		Now
		\[\frac{cx+dg(x)}{ax+bg(x)}=d+\frac{(c-ad)x}{ax+g(x)}.\]
		As $M$ is nonsingular, $c-ad\neq 0$. Furthermore, 
%		by definition of maximum scattered linear sets in $\PG(1,q^{2m})$, 
		the map $x\mapsto ax+g(x)$ must be invertible; otherwise the point $\langle (0,1)\rangle_{\F_{q^{2m}}}$ would be in $L_f$ which is impossible.
		
		Let $ h(x)=\sum_{i=0}^{2m-1} r_i x^{q^i} $ denote the inverse map of $x\mapsto \frac{ ax+g(x)}{c-ad}$. Setting $y=\frac{ ax+g(x)}{c-ad}$, we get
		\[
		d+\frac{(c-ad)x}{ax+g(x)}=d+\frac{h(y)}{y}.
		\]
		By comparing the coefficients of $ah(Y)+g(h(Y))\equiv (c-ad)Y \pmod{Y^{q^{2m}}-Y}$, we obtain
%		\begin{equation}\label{eq:d=0_2m_1}
%			a r_0 +r_{2m-s_2}^{q^{s_2}}+\theta r_{m-s_2}^{q^{s_2+m}} = c-ad
%		\end{equation} 
%		and
		\begin{equation}\label{eq:d=0_2m_2}
			a r_i +r_{i-s_2}^{q^{s_2}}+\theta r_{i-s_2+m}^{q^{s_2+m}} = 0,
		\end{equation} 	
		for $i=1,2,\cdots,2m-1$.
		Now \eqref{eq:d=0} becomes
		\begin{equation*}
%		\label{eq:d=0_b=1}
			\left\{ x^{q^{s_1}-1}+\delta x^{q^{s_1+m}-1} : x\in \F_{q^{2m}}^*  \right\} = \left\{ d+\sum_{i=0}^{2m-1}r_i y^{q^i-1}: y\in \F_{q^{2m}}^*  \right\}.
		\end{equation*}
		We apply Lemma \ref{lem1} on the above equation. By \eqref{eq:a_0_b_0},
		\begin{equation}\label{eq:d=0_3}
			r_0=-d.
		\end{equation}
		Moreover, by \eqref{eq:k_1_n-1},
		\begin{equation}\label{eq:d=0_4}
			r_k r_{2m-k}^{q^k}=0,
		\end{equation}
		for $k=1,2,\cdots,2m-1$.
		By \eqref{eq:d=0_4} with $k=s_2$, 
		\[ r_{s_2} r_{2m-s_2}^{q^{s_2}}=0,\]
		which means at least one of $r_{s_2}$ and $r_{2m-s_2}$ equals $0$. Without loss of generality, we assume that $r_{s_2}=0$. 
		By \eqref{eq:d=0_4} with $k=m$,  $r_m^{q^m+1}=0$, which means 
		\begin{equation}\label{eq:d=0_5}
			r_m=0.
		\end{equation}
		By \eqref{eq:d=0_2m_2} with $i=s_2$, that is 
		\begin{equation}\label{eq:d=0_6}
			r_0^{q^{s_2}}+\theta r_m^{q^{s_2+m}}=0,
		\end{equation}
		By \eqref{eq:d=0_3}, \eqref{eq:d=0_5} and \eqref{eq:d=0_6}, we get $d=-r_0=0$.
		
		As $$\begin{pmatrix}a & 1   \\ c & 0 \end{pmatrix}^{-1}=\begin{pmatrix}0 & 1/c   \\ 1 & -a/c \end{pmatrix},$$
		by considering the inverse of $M$ acting on $L_g$, we get $a = 0$.
	\end{proof}

Now we are ready to consider the equivalence problem for $\CFA$. The following theorem completely determine the equivalence between different members in $\CFA$. 
\begin{theorem}\label{th:LP_n_6}
	Let $f(x)=x^{q}+\delta x^{q^4}$ and $g(x)=x^{q}+\theta x^{q^4}$. %with $\delta,\theta\in \F_{q^6}^*$ and $N_{q^{6}/{q^3}}(\delta ), N_{q^{6}/{q^3}}(\theta )\neq 1$.
%	\zhou{In fact, if $N_{q^{6}/{q^3}}(\delta ), N_{q^{6}/{q^3}}(\theta ) =1$, the result still holds and the only difference of the proof is Case 2 in which the inverse of $g$ does not exist.} 
%	Assume that $L_f$ and $L_g$ are maximum scattered linear sets over $\PG(1,q^6)$. 
	Then  $L_f$ and $L_g$ are 
	$\PGamL$-equivalent if and only if  
	\begin{equation}\label{s_1,s_2_6}
		N_{q^6/{q^3}}(\delta ) = N_{q^6/{q^3}}(\theta )^\sigma,
	\end{equation}
	for some automorphism $\sigma\in \mathrm{Aut}(\F_{q^{3}})$.
\end{theorem}

\begin{proof}
	If $\theta=\delta=0$,  the above conclusion clearly holds. Hence, we can assume that at least one of $\theta$ and $\delta$ is not equal to $0$.

	First,  we prove the necessity of \eqref{s_1,s_2_6}. 
	
	Assume that $L_f$ and $L_g$ are equivalent which implies the existence of invertible matrix $M =\begin{pmatrix}
	a & b \\ 
	c & d
	\end{pmatrix} $ over $\F_{q^6}$ and $\tau\in\Aut(\F_{q^3})$ such that
	\begin{equation}\label{eq:LP_equivalence_n_6}
	\left\{ x^{q-1}+\delta x^{q^4-1} : x\in \F_{q^6}^*  \right\} = \left\{ \frac{cx+d\bar{g}(x)}{ax+b\bar{g}(x)}: x\in \F_{q^6}^*  \right\},
	\end{equation}
	where $\bar{g}(x)=x^{q}+\theta^\tau x^{q^4} $.
	
	Depending on the value of $b$, we separate the rest part into two cases.
	
	\medskip
	
	\noindent\textbf{Case 1:} $b=0$.  As $M$ is viewed as an element in $\PGL(2,q^6)$, we can assume that $a=1$ and \eqref{eq:LP_equivalence_n_6} becomes
	\begin{equation}\label{eq:LP_equi_n=6}
	\left\{ x^{q-1}+\delta x^{q^4-1}  : x\in \F_{q^6}^*  \right\} = \left\{ c+d\frac{\bar{g}(x)}{x}: x\in \F_{q^6}^*  \right\}.
	\end{equation}
	By \eqref{eq:a_0_b_0} in Lemma \ref{lem1}, we derive $c=0$.
	Then \eqref{eq:LP_equi_n=6} becomes
	$$	\left\{ x^{q-1}+\delta x^{q^4-1}  : x\in \F_{q^6}^*  \right\}=	\left\{d\left(x^{q-1}+\theta^\tau x^{q^4-1}\right)  : x\in \F_{q^6}^*  \right\}.$$
	By \eqref{eq:k_2_n-1} with $k=2$ in Lemma \ref{lem1}, we get 
	\begin{equation}
	\label{eq:b_0_1}
	d^{q^2+q+1}\theta^{\tau q^2}=\delta^{q^2}.
	\end{equation}
	If  exactly one of $\theta$ and $\delta$ equals $0$,  then it has to be $\delta$ and $d$ must be $0$. However, this implies $M$ is singular which leads to a contradiction. Hence $\theta$ and $\delta\neq 0$.
	By \eqref{eq:(2,0)_1,1} in Lemma \ref{D_6}, we get 
	\begin{equation}
	\label{eq:b_0_2}
	d^{(q^2-q+1)(q^2+q+1)}=\frac{\delta^{q^4+q^2+1}}{\theta^{\tau\left(q^4+q^2+1\right)}}.
	\end{equation}
	Then by \eqref{eq:b_0_1} and \eqref{eq:b_0_2}, we get 
	\begin{equation*}
	\delta^{q^3+1}=\left(\theta^{q^3+1}\right)^\tau,
	\end{equation*}
let $\sigma=\tau|_{\F_{q^3}}$, then
	\begin{equation*}
	N_{q^{6}/{q^3}}(\delta) = N_{q^{6}/{q^3}}(\theta )^\sigma.
	\end{equation*}
	
	\medskip
	\noindent\textbf{Case 2:} $b\neq 0$. Without loss of generality, we assume that $b=1$. By Lemma \ref{d=0}, we get $a=d=0$, which means
	$$M=\begin{pmatrix}0 & 1   \\ c & 0 \end{pmatrix}.$$
	
	Let $h_0(x)=-\theta^{\tau q^5}x^{q^2}+x^{q^5}$, then $h_0(\bar{g}(x))=\left(1-\theta^{\tau(q^2+q^5)}\right)x$. 
	
	If  $N_{q^{6}/{q^3}}(\theta )=1$, which is equivalent to $1-\theta^{\tau(q^2+q^5)}=0$, then the map $x\mapsto g(x)$ must be not invertible. By \eqref{eq:LP_equivalence_n_6},  the point $\langle (0,1)\rangle_{\F_{q^{6}}}$ would be in $L_f$ which is impossible. Hence  $N_{q^{6}/{q^3}}(\theta )\neq 1 $.
	
	Let $$h(x)=\frac{h_0(x)}{1-\theta^{\tau(q^2+q^5)}}=-\frac{\theta^{\tau q^5}x^{q^2}}{1-\theta^{\tau(q^2+q^5)}}+\frac{x^{q^5}}{1-\theta^{\tau(q^2+q^5)}}.$$
	Then $h(\bar{g}(x))=x$, i.e. $h(x)$ is the inverse map of $\bar{g}(x)$.  Thus \eqref{eq:LP_equivalence_n_6} equals 
	\begin{equation}\label{eq:LP_U_6_b_1_2}
	\left\{  x^{q-1}+\delta x^{q^4-1} :x\in \mathbb{F}_{q^6}^*\right\}=	\left\{\frac{ch(x)}{x}:x\in \mathbb{F}_{q^6}^*\right\}.
	\end{equation}
	By a simple calculation, the adjoint $\hat{h}_1(x)$ of $h_1(x)$ is
	$$\hat{h}_1(x)= x^q + (-\theta^{q^5\tau})^{q^4}x^{q^4},$$ where $h_1(x)=h(x)(1-\theta^{\tau(q^2+q^5)})$.	
	By Lemma \ref{le:LS_adjoint}, $L_{h_1}=L_{\hat{h}_1}$, then \eqref{eq:LP_U_6_b_1_2} equals
	\begin{equation}\label{eq:LP_U_6_b_1_3}
	\left\{  x^{q-1}+\delta x^{q^4-1} :x\in \mathbb{F}_{q^6}^*\right\}=\left\{ \hat{c} \left(x^{q-1}+(-\theta^{q^5\tau})^{q^4} x^{q^4-1}\right) :x\in \mathbb{F}_{q^6}^*\right\},
	\end{equation}
	where $\hat{c}=\frac{c}{1-\theta^{\tau(q^2+q^5)}}$.
	
	By \eqref{eq:k_2_n-1} in Lemma \ref{lem1}, we get 
	\begin{equation}
	\label{eq:b_1_1}
	\hat{c}^{q^2+q+1}(-\theta^{q^5\tau})=\delta^{q^2}.
	\end{equation}
	If  exactly one of $\theta$ and $\delta$ equals $0$,  then it has to be $\delta$ and $c$ must be $0$ which means $M$ is singular. However, this contradicts our assumption on $M$. Hence $\theta$ and $\delta\neq 0$.
	By \eqref{eq:(2,0)_1,1} in Lemma \ref{D_6}, we get 
	\begin{equation}
	\label{eq:b_1_2}
	\hat{c}^{(q^2-q+1)(q^2+q+1)}=\frac{\delta^{q^4+q^2+1}}{(-\theta^{q^5\tau})^{q^4\left(q^4+q^2+1\right)}}.
	\end{equation}
	Then by \eqref{eq:b_1_1} and \eqref{eq:b_1_2}, we get 
	\begin{equation*}
	\delta^{q^3+1}=\left(-\theta^{q^5\tau}\right)^{q^4(q^3+1)},
	\end{equation*}
	let $\sigma=\tau|_{\F_{q^3}}$, then
	\begin{equation*}
	N_{q^{6}/{q^3}}(\delta) = N_{q^{6}/{q^3}}(\theta )^\sigma.
	\end{equation*}
	Therefore, we have finished the proof of the necessity of \eqref{s_1,s_2_6}. 
	
	Next, we prove the sufficiency of \eqref{s_1,s_2_6}. Assume that 
	$	N_{q^{6}/{q^3}}(\delta) = N_{q^{6}/{q^3}}(\theta )^\tau$ and $\theta, \delta \neq 0$.
	
	We only have to prove that there exists an invertible matrix $M$ such that \eqref{eq:LP_equivalence_n_6} holds. Let  $M =\begin{pmatrix}
	1 & 0 \\ 
	0 & d
	\end{pmatrix} $ where $d$ satisfies \eqref{eq:b_0_1}. Hence $$d^{(q^3+1)(q^2+q+1)}= \left(\frac{\delta}{\theta^\tau}\right)^{q^2(q^3+1)}=1.$$
	and there exists $\alpha \in \F_{q^6}^*$ such that $d=\alpha^{q-1}$. Replacing $x$ by $\alpha x$, we obtain
	\begin{equation}\label{eq:N_1}
	L_f=\left\{ (\alpha x)^{q-1}+\delta (\alpha x)^{q^4-1} : x\in \F_{q^6}^*  \right\} = \left\{ \alpha^{q-1}\left(x^{q-1}+ \alpha^{q^4-q} \delta x^{q^4-1}\right): x\in \F_{q^6}^*  \right\}.
	\end{equation}	
	As $\alpha^{q^4-q} \delta=d^{q(q^2+q+1)} \delta=\frac{\delta^{q^3+1}}{\theta^{q^3\tau}}=\theta^\tau$ and $d=\alpha^{q-1}$, \eqref{eq:N_1} is equivalent to
	\begin{equation*}
	L_f = \left\{ d\left(x^{q-1}+ \theta^\tau x^{q^4-1}\right): x\in \F_{q^6}^*  \right\}=L_{d\bar{g}}.\qedhere
	\end{equation*}
\end{proof}

\begin{remark}
	In fact, if $N_{q^6/{q^3}}(\delta ) = N_{q^6/{q^3}}(\theta )^\sigma \neq 1$, we can use Lemma \ref{lem:s_1,s_2} directly to prove the sufficient condition in Theorem \ref{th:LP_n_6}.
\end{remark}

\begin{theorem}\label{th:Aut_6}
		Let $f(x)=x^q+\theta x^{q^4}$ over $\F_{q^6}$ with $\theta\neq 0$. 
%		If $L_f$ is a maximum scattered linear set, 
		The automorphism group of $L_f$ is		
	\begin{align*}
			\Aut(L_f)=\begin{cases}
				\mathcal{D}, &N_{q^6/{q^3}}(\theta )=1;\\
				\mathcal{D}\cup\mathcal{C}, & N_{q^6/{q^3}}(\theta )\neq 1,
			\end{cases}
		\end{align*}
	where 	$$\mathcal{D}=\left\{M\tau : 	N_{q^{6}/{q^3}}(\theta)^{\tau-1} = 1, M=
	\begin{pmatrix}
		1 & 0 \\ 
		0 & d
	\end{pmatrix}\in\mathrm{PGL}(2,q^6), \tau\in \mathrm {Aut}(\F_{q^6}),	d^{q^2+q+1}=\left(\frac{\theta}{\theta^\tau} \right)^{q^2}\right\},$$
and 	$$\mathcal{C}=\left\{	M\tau : 
	N_{q^{6}/{q^3}}(\theta)^{\tau -1}=1, 	M=\begin{pmatrix}
		0 & 1 \\ 
		c & 0
	\end{pmatrix}\in\mathrm{PGL}(2,q^6), \tau\in \mathrm {Aut}(\F_{q^6}), \left(\frac{c}{1-\theta^{\tau(q^2+q^5)}}\right)^{q^2+q+1}=\frac{\theta^{q^2}}{-\theta^{\tau q^5}}\right\}.$$
	\end{theorem}
	\begin{proof}
			We only have to follow the proof of Theorem \ref{th:LP_n_6} under the assumption that $\theta=\delta$.
				 Depending on the value of $b$, we consider two cases.
				
				\medskip
				
				\noindent\textbf{Case 1:} $b=0$.
				When $b=0$ in \eqref{eq:LP_equivalence_n_6}, we always have $a=1$ and $c=0$. Hence, we just need to determine $d\in \F_{q^6}$ and $\tau\in \Aut(\F_{q^6})$ such that $L_{d\bar{f}}=L_f$ where $\bar{f}(x)=x^{q}+\theta^\tau x^{q^4}$. By \eqref{eq:b_0_1} and \eqref{eq:b_0_2}, if $L_{d\bar{f}}=L_f$, then $d,\tau$ satisfy
				\begin{equation}
					\label{eq:aut_6_b_0_1}	d^{q^2+q+1}=\left(\frac{\theta}{\theta^\tau}\right)^{q^2},
				\end{equation}
				and 
				\begin{equation}
					\label{eq:aut_6_b_0_2}
					N_{q^{6}/{q^3}}(\theta) = N_{q^{6}/{q^3}}(\theta )^\tau.
				\end{equation}
					Thus  \eqref{eq:aut_6_b_0_1} and \eqref{eq:aut_6_b_0_2} are  the necessary condition.	
					
					Next, we show that \eqref{eq:aut_6_b_0_1} and \eqref{eq:aut_6_b_0_2} are also sufficient. By \eqref{eq:aut_6_b_0_1} and \eqref{eq:aut_6_b_0_2}, we get 
					$$d^{(q^3+1)(q^2+q+1)}=\left(\frac{\theta}{\theta^\tau}\right)^{q^2(q^3+1)}=N_{q^{6}/{q^3}}\left(\frac{\theta}{\theta^\tau}\right)^{q^2}=1.$$
						Hence there exists $\alpha \in \F_{q^6}^*$ such that $d=\alpha^{q-1}$. Replacing $x$ by $\alpha x$, we obtain
					\begin{equation}
						\label{eq:alpha_2}
						\left\{ (\alpha x)^{q-1}+\theta (\alpha x)^{q^4-1} : x\in \F_{q^6}^*  \right\} = \left\{ \alpha^{q-1}\left(x^{q-1}+ \alpha^{q^4-q} \theta x^{q^4-1}\right): x\in \F_{q^6}^*  \right\}.
					\end{equation}	
					As $\alpha^{q^4-q} \theta=d^{q(q^2+q+1)} \theta=\frac{\theta^{q^3+1}}{\theta^{q^3\tau}}=\theta^\tau$, \eqref{eq:alpha_2} is equivalent to
				\begin{equation*}
					\left\{ (\alpha x)^{q-1}+\theta (\alpha x)^{q^4-1} : x\in \F_{q^6}^*  \right\} = \left\{ d\left(x^{q-1}+ \theta^\tau x^{q^4-1}\right): x\in \F_{q^6}^*  \right\}.
				\end{equation*}

			\medskip
			
			\noindent\textbf{Case 2:} $b\neq0$.
				When $b\neq 0$ in \eqref{eq:LP_equivalence_n_6}, we have $N_{q^{6}/{q^3}}(\theta)\neq 1$, $b=1$, and $a=d=0$. Hence, we just need to determine $c\in \F_{q^6}$ and $\tau\in \Aut(\F_{q^6})$ such that $L_{ch}=L_f$ where \[h(x)=-\frac{\theta^{\tau q^5}x^{q^2}}{1-\theta^{\tau(q^2+q^5)}}+\frac{x^{q^5}}{1-\theta^{\tau(q^2+q^5)}}.\]
				By \eqref{eq:b_1_1} and \eqref{eq:b_1_2}, if $L_{ch}=L_f$, then $c,\tau$ satisfy
				\begin{equation}
					\label{eq:aut_6_b_1_1}
					\hat{c}^{q^2+q+1}=\frac{\theta^{q^2}}{-\theta^{q^5\tau}},
				\end{equation}
				where $\hat{c}=\frac{c}{1-\theta^{\tau(q^2+q^5)}}$, and 
				\begin{equation}
					\label{eq:aut_6_b_1_2}
					N_{q^{6}/{q^3}}(\theta) = N_{q^{6}/{q^3}}(\theta )^\tau.
				\end{equation}
				Hence  \eqref{eq:aut_6_b_1_1} and \eqref{eq:aut_6_b_1_2} are  the necessary condition.
				
				Next, we show that \eqref{eq:aut_6_b_1_1} and \eqref{eq:aut_6_b_1_2} are also sufficient.
				By \eqref{eq:LP_U_6_b_1_2}, \eqref{eq:LP_U_6_b_1_3} and $\delta=\theta$, we only need to proof that the sufficiency for \eqref{eq:LP_U_6_b_1_3} with $\delta=\theta$.
				By \eqref{eq:aut_6_b_1_1} and \eqref{eq:aut_6_b_1_2}, we get 
		$$\hat{c}^{(q^3+1)(q^2+q+1)}=\left(\frac{\theta^{q^2}}{-\theta^{q^5\tau}}\right)^{q^3+1}=\frac{N_{q^{6}/{q^3}}(\theta)^{q^2} }{N_{q^{6}/{q^3}}(\theta )^{q^2\tau}}=1.$$
		Hence there exists $\alpha \in \F_{q^6}^*$ such that $\hat{c}=\alpha^{q-1}$. Replacing $x$ by $\alpha x$, we obtain
		\begin{equation}
			\label{eq:alpha_3}
			\left\{ (\alpha x)^{q-1}+\theta (\alpha x)^{q^4-1} : x\in \F_{q^6}^*  \right\} = \left\{ \alpha^{q-1}\left(x^{q-1}+ \alpha^{q^4-q} \theta x^{q^4-1}\right): x\in \F_{q^6}^*  \right\}.
		\end{equation}
		As $\alpha^{q^4-q} \theta=\hat{c}^{q(q^2+q+1)} \theta=\frac{\theta^{q^3+1}}{(-\theta^{q^5\tau})^q}=(-\theta^{q^5\tau})^{q^4}$, \eqref{eq:alpha_3} is equivalent to
		\begin{equation*}
			\left\{ (\alpha x)^{q-1}+\theta (\alpha x)^{q^4-1} : x\in \F_{q^6}^*  \right\} = \left\{ \hat{c}\left(x^{q-1}+ (-\theta^{q^5\tau})^{q^4} x^{q^4-1}\right): x\in \F_{q^6}^*  \right\}.\qedhere
		\end{equation*}
	\end{proof}
	
	\begin{remark}\label{re:F6_equiv}
		It is worth pointing out that $L_f$ and $L_g$ are not required to be  maximum scattered in Theorem \ref{th:LP_n_6} and Theorem \ref{th:Aut_6}. In another word, we have determined the equivalence and the automorphism groups of the members in a larger family of linear sets in $\PG(1,q^6)$ which contains the family $\CFA$ of maximum scattered ones discovered by Csjab\'ok, Marino, Polverino and Zanella in \cite{csajbok_classes_2018}.
	\end{remark}
	
	Next we turn to the investigation of equivalence problem for the second family $\CFB$ constructed in \cite{csajbok_classes_2018}.
	\begin{theorem}\label{th:LP_n_8}
		Let $f(x)=x^{q}+\delta x^{q^5}$ and $g(x)=x^{q}+\theta x^{q^5}$. % \myue{??with $N_{q^{8}/{q^4}}(\delta), N_{q^{8}/{q^4}}(\theta ) \notin \{0,1\} $, $\delta$ and $\theta$ can be $0$ together}. If  $L_f,L_g$ are maximum scattered linear sets , then 
		The linear sets $L_f$ and $L_g$ are 
		$\PGamL-$equivalent is if and only if  
		\begin{equation}\label{s_1,s_2_8}
			N_{q^8/{q^4}}(\delta ) = N_{q^8/{q^4}}(\theta )^\sigma,
		\end{equation}
		for some automorphism $\sigma\in \mathrm{Aut}(\F_{q^{4}})$.
	\end{theorem}
	\begin{proof}
			If $\theta=\delta=0$,  the above conclusion is clearly correct. Hence, we can assume that at least one of $\theta$ and $\delta$ is not equal to $0$.
	
	First,  we prove the necessity of \eqref{s_1,s_2_8}.
	
	Assume that $L_f$ and $L_g$ are equivalent which implies the existence of invertible matrix $M =\begin{pmatrix}
	a & b \\ 
	c & d
	\end{pmatrix} $ over $\F_{q^8}$ and $\tau\in\Aut(\F_{q^4})$ such that
	\begin{equation}\label{eq:LP_equivalence_n_8}
	\left\{ x^{q-1}+\delta x^{q^5-1} : x\in \F_{q^8}^*  \right\} = \left\{ \frac{cx+d\bar{g}(x)}{ax+b\bar{g}(x)}: x\in \F_{q^{2m}}^*  \right\},
	\end{equation}
	where $\bar{g}(x)=x^{q}+\theta^\tau x^{q^5}$.
	
	Next, we investigate \eqref{eq:LP_equivalence_n_8} in two different cases.
	
	\medskip
	
	\noindent\textbf{Case 1:} $b=0$.  As $M$ is viewed as an element in $\PGL(2,q^8)$, we can assume that $a=1$.
	By \eqref{eq:a_0_b_0} in Lemma \ref{lem1}, we derive $c=0$.
	Then \eqref{eq:LP_equivalence_n_8} becomes
	$$	\left\{ x^{q-1}+\delta x^{q^5-1}  : x\in \F_{q^8}^*  \right\}=	\left\{d\left(x^{q-1}+\theta^\tau x^{q^5-1}\right)  : x\in \F_{q^8}^*  \right\}.$$
	By \eqref{eq:(3,2,1,0)_1,1} in Lemma \ref{D_8}, we get 
	\begin{equation}
	\label{eq:b_0_1_n_8}
	d^{1+q+q^2+q^3}\theta ^{\tau q^3}=\delta^{q^3},
	\end{equation}
	If  exactly one of $\theta$ and $\delta$ equals $0$,  then it has to be $\delta$ and $d$ must be $0$. However, this implies $M$ is singular which leads to a contradiction. Hence $\theta$ and $\delta\neq 0$.
	
	By \eqref{eq:(3,2,0)(3,2,1,0)_1,1} in Lemma \ref{D_8}, 
	we get 	
	\begin{equation}
	\label{eq:b_0_2_n_8}
	d^{(q^3-q^2+1)(1+q+q^2+q^3)}=\frac{\delta^{q^6+q^3+q}}{\theta ^{\tau (q^6+q^3+q)}}.
	\end{equation}
	Then by \eqref{eq:b_0_1_n_8} and \eqref{eq:b_0_2_n_8}, we get 
	$$\delta^{q^5+q}=\theta^{\tau (q^5+q)},$$
let $\sigma=\tau|_{\F_{q^4}}$, then
	$$	N_{q^{8}/{q^4}}(\delta)= N_{q^{8}/{q^4}}(\theta )^{\tau}.$$
	
	\medskip
	\noindent\textbf{Case 2:} $b\neq 0$. Without loss of generality, we assume that $b=1$. By Lemma \ref{d=0}, we get $a=d=0$ which implies 
	$$M=\begin{pmatrix}0 & 1   \\ c & 0 \end{pmatrix}.$$
	
	Let $h_0(x)=-\theta^{\tau q^7}x^{q^3}+x^{q^7}$, then $h_0(\bar{g}(x))=\left(1-\theta^{\tau(q^3+q^7)}\right)x$. 
	
	If $N_{q^{8}/{q^4}}(\theta )=1$ which is equivalent to $1-\theta^{\tau(q^3+q^7)}=0$, then the map $x\mapsto g(x)$ must be not invertible. By \eqref{eq:LP_equivalence_n_8},  the point $\langle (0,1)\rangle_{\F_{q^{8}}}$ would be in $L_f$ which is impossible. Hence  $N_{q^{8}/{q^4}}(\theta )\neq 1 $.
	
	Let $$h(x)=-\frac{\theta^{\tau q^7}x^{q^3}}{1-\theta^{\tau(q^3+q^7)}}+\frac{x^{q^7}}{1-\theta^{\tau(q^3+q^7)}}.$$
	Then $h(\bar{g}(x))=x$, i.e. $h(x)$ is the inverse map of $\bar{g}(x)$ with $\bar{g}(x)=x^q+\theta^\tau x^{q^5}$.
	Thus \eqref{eq:LP_equivalence_n_8} equals 
	\begin{equation}\label{eq:b_1_1_n_8}
	\left\{ x^{q-1}+\delta x^{q^5-1} :x\in \mathbb{F}_{q^8}^*\right\}=\left\{\frac{ch(x)}{x}:x\in \mathbb{F}_{q^8}^*\right\}.
	\end{equation}
	By a simple calculation, the adjoint $\hat{h}_1(x)$ of $h_1(x)$ is
	$$\hat{h}_1(x)= x^q + (-\theta^{q^7\tau})^{q^5}x^{q^5},$$ where $h_1(x)=(1-\theta^{\tau(q^3+q^7)})h(x)$.
	
	By Lemma \ref{le:LS_adjoint}, $L_{h_1}=L_{\hat{h}_1}$, \eqref{eq:b_1_1_n_8} equals
	\begin{equation}\label{eq:b_1_2_n_8}
	\left\{  x^{q-1}+\delta x^{q^5-1} :x\in \mathbb{F}_{q^8}^*\right\}=\left\{  \hat{c}\left(x^{q-1}+(-\theta^{q^7\tau})^{q^5} x^{q^5-1}\right) :x\in \mathbb{F}_{q^8}^*\right\},
	\end{equation}
	where $\hat{c}=\frac{c}{1-\theta^{\tau(q^3+q^7)}}$.
	
	By \eqref{eq:(3,2,1,0)_1,1} in Lemma \ref{D_8}, we get 
	\begin{equation}
	\label{eq:b_1_3_n_8}
	\hat{c}^{1+q+q^2+q^3}(-\theta^{q^7\tau})=\delta^{q^3},
	\end{equation}
	If one of $\theta$ and $\delta$ equals $0$,  then it has to be $\delta$ from which it follows $\hat{c}=0$. It contradicts  the assumption that $M$ is invertible. Hence $\theta$ and  $\delta$ must be nonzero.
	
	Furthermore, by \eqref{eq:(3,2,0)(3,2,1,0)_1,1} in Lemma \ref{D_8}, 
	we get 	
	\begin{equation}
	\label{eq:b_1_4_n_8}
	\hat{c}^{(q^3-q^2+1)(1+q+q^2+q^3)}=\frac{\delta^{q^6+q^3+q}}{(-\theta^{q^7\tau})^{q^5 (q^6+q^3+q)}}.
	\end{equation}
	Then by \eqref{eq:b_1_3_n_8} and \eqref{eq:b_1_4_n_8}, we get 
	$$\delta^{q^5+q}=(-\theta^{q^7\tau})^{q^5(q^5+q)},$$
let $\sigma=\tau|_{\F_{q^4}}$, then
	\begin{equation*}
	N_{q^{8}/{q^4}}(\delta)= N_{q^{8}/{q^4}}(\theta )^{\tau}.
	\end{equation*}
	Therefore, we have finished the proof of the necessity of \eqref{s_1,s_2_8}.
	
	Now, we prove the sufficiency of \eqref{s_1,s_2_8}. Assume that 
	$N_{q^{8}/{q^4}}(\delta)= N_{q^{8}/{q^4}}(\theta )^{\tau}$ with $\theta, \delta \neq 0$.
	
	We only have to prove that there exists an invertible matrix $M$ such that \eqref{eq:LP_equivalence_n_8} holds. Let  $M =\begin{pmatrix}
	1 & 0 \\ 
	0 & d
	\end{pmatrix} $ where $d$ satisfies \eqref{eq:b_0_1_n_8}. Hence $$d^{(q^4+1)(q^3+q^2+q+1)}= \left(\frac{\delta}{\theta^\tau}\right)^{q^3(q^4+1)}=1.$$
	and there exists $\alpha \in \F_{q^8}^*$ such that $d=\alpha^{q-1}$. Replacing $x$ by $\alpha x$, we obtain
	\begin{equation}\label{eq:N_2}
	L_f=\left\{ (\alpha x)^{q-1}+\delta (\alpha x)^{q^5-1} : x\in \F_{q^8}^*  \right\} = \left\{ \alpha^{q-1}\left(x^{q-1}+ \alpha^{q^5-q} \delta x^{q^5-1}\right): x\in \F_{q^8}^*  \right\}.
	\end{equation}	
	As $\alpha^{q^5-q} \delta=d^{q(q^3+q^2+q+1)} \delta=\frac{\delta^{q^3+1}}{\theta^{q^3\tau}}=\theta^\tau$ and $d=\alpha^{q-1}$, \eqref{eq:N_2} is equivalent to
	\begin{equation*}
	L_f = \left\{ d\left(x^{q-1}+ \theta^\tau x^{q^5-1}\right): x\in \F_{q^8}^*  \right\}=L_{d\bar{g}}.\qedhere
	\end{equation*}
	\end{proof}

	\begin{theorem}\label{th:Aut_8}
		Let $f(x)=x^q+\theta x^{q^5}$ over $\F_{q^8}$ with $\theta\neq 0$. 
	 	The automorphism group of $L_f$ is 
	\begin{align*}
				\Aut(L_f)=\begin{cases}
					\mathcal{D}, &N_{q^8/{q^4}}(\theta )=1;\\
					\mathcal{D}\cup\mathcal{C}, & N_{q^8/{q^4}}(\theta )\neq 1,
				\end{cases}
	\end{align*}
	where 	$$\mathcal{D}=\left\{M\tau : 	N_{q^{8}/{q^4}}(\theta)^{\tau-1} = 1, M=
		\begin{pmatrix}
			1 & 0 \\ 
			0 & d
		\end{pmatrix}\in\mathrm{PGL}(2,q^8), \tau\in \mathrm {Aut}(\F_{q^8}),	d^{q^3+q^2+q+1}=\left(\frac{\theta}{\theta^\tau} \right)^{q^3}\right\},$$
		and
		$$\mathcal{C}=\left\{	M\tau : 
		N_{q^8/{q^4}}(\theta)^{\tau -1}=1, 	M=\begin{pmatrix}
			0 & 1 \\ 
			c & 0
		\end{pmatrix}\in\mathrm{PGL}(2,q^8), \tau\in \mathrm {Aut}(\F_{q^8}), \left(\frac{c}{1-\theta^{\tau(q^3+q^7)}}\right)^{q^3+q^2+q+1}=\frac{\theta^{q^3}}{-\theta^{\tau q^7}}\right\}.$$
	\end{theorem}
	\begin{proof}
			We only have to follow the proof of Theorem \ref{th:LP_n_8} under the assumption that $\theta=\delta$.
		Depending on the value of $b$, we separate the rest part into two cases.
		
		\medskip
		
		\noindent\textbf{Case 1:} $b=0$.
		When $b=0$ in \eqref{eq:LP_equivalence_n_8}, we always have $a=1$ and $c=0$. Hence, we just need to determine $d\in \F_{q^8}$ and $\tau\in \Aut(\F_{q^8})$ such that $L_{d\bar{f}}=L_f$ where $\bar{f}(x)=x^{q}+\theta^\tau x^{q^5}$. By \eqref{eq:b_0_1_n_8} and \eqref{eq:b_0_2_n_8}, if $L_{d\bar{f}}=L_f$, then $d,\tau$ satisfy
		\begin{equation}
			\label{eq:aut_8_b_0_1}	d^{q^3+q^2+q+1}=\left(\frac{\theta}{\theta^\tau}\right)^{q^3},
		\end{equation}
		and 
		\begin{equation}
			\label{eq:aut_8_b_0_2}
			N_{q^{8}/{q^4}}(\theta) = N_{q^{8}/{q^4}}(\theta )^\tau.
		\end{equation}
		Then  \eqref{eq:aut_8_b_0_1} and \eqref{eq:aut_8_b_0_2} are  the necessary condition.
		
		Next, we show that \eqref{eq:aut_8_b_0_1} and \eqref{eq:aut_8_b_0_2} are also sufficient. By \eqref{eq:aut_8_b_0_1} and \eqref{eq:aut_8_b_0_2}, we get 
		$$d^{(q^4+1)(q^3+q^2+q+1)}=\left(\frac{\theta}{\theta^\tau}\right)^{q^3(q^4+1)}=N_{q^{8}/{q^4}}\left(\frac{\theta}{\theta^\tau}\right)^{q^3}=1.$$
		Hence there exists $\alpha \in \F_{q^8}^*$ such that $d=\alpha^{q-1}$. Replacing $x$ by $\alpha x$, we obtain
		\begin{equation}
			\label{eq:alpha_4}
			\left\{ (\alpha x)^{q-1}+\theta (\alpha x)^{q^5-1} : x\in \F_{q^8}^* \right\} = \left\{ \alpha^{q-1}\left(x^{q-1}+ \alpha^{q^5-q} \theta x^{q^5-1}\right): x\in \F_{q^8}^*  \right\}.
		\end{equation}	
		As $\alpha^{q^5-q} \theta=d^{q(q^3+q^2+q+1)} \theta=\frac{\theta^{q^4+1}}{\theta^{q^4\tau}}=\theta^\tau$, \eqref{eq:alpha_4} is equivalent to
		\begin{equation*}
			\left\{ (\alpha x)^{q-1}+\theta (\alpha x)^{q^5-1} : x\in \F_{q^8}^*  \right\} = \left\{ d\left(x^{q-1}+ \theta^\tau x^{q^5-1}\right): x\in \F_{q^8}^*  \right\}.
		\end{equation*}

		\medskip
		
		\noindent\textbf{Case 2:} $b\neq0$.
		When $b\neq 0$ in \eqref{eq:LP_equivalence_n_8}, we have $N_{q^{8}/{q^4}}(\theta)\neq 1$, $b=1$, and $a=d=0$. Hence, we just need to determine $c\in \F_{q^8}$ and $\tau\in \Aut(\F_{q^8})$ such that $L_{ch}=L_f$ where  $h(x)=-\frac{\theta^{\tau q^7}x^{q^3}}{1-\theta^{\tau(q^3+q^7)}}+\frac{x^{q^7}}{1-\theta^{\tau(q^3+q^7)}}$. By \eqref{eq:b_1_3_n_8} and \eqref{eq:b_1_4_n_8}, if $L_{ch}=L_f$, then $c,\tau$ satisfy
		\begin{equation}
			\label{eq:aut_8_b_1_1}
			\hat{c}^{q^3+q^2+q+1}=\frac{\theta^{q^3}}{-\theta^{q^7\tau}},
		\end{equation}
		where $\hat{c}=\frac{c}{1-\theta^{\tau(q^3+q^7)}}$, and 
		\begin{equation}
			\label{eq:aut_8_b_1_2}
			N_{q^{8}/{q^4}}(\theta) = N_{q^{8}/{q^4}}(\theta )^\tau.
		\end{equation}
		Thus \eqref{eq:aut_8_b_1_1} and \eqref{eq:aut_8_b_1_2} are  the necessary condition.	
		
		Next, we show that \eqref{eq:aut_8_b_1_1} and \eqref{eq:aut_8_b_1_2} are also sufficient.
		By \eqref{eq:b_1_1_n_8}, \eqref{eq:b_1_2_n_8} and $\delta=\theta$, we only need to proof that the sufficiency for \eqref{eq:b_1_2_n_8} with $\delta=\theta$.
		By \eqref{eq:aut_8_b_1_1} and \eqref{eq:aut_8_b_1_2}, we get 
		$$\hat{c}^{(q^4+1)(q^3+q^2+q+1)}=\left(\frac{\theta^{q^3}}{-\theta^{q^7\tau}}\right)^{q^4+1}=\frac{N_{q^{8}/{q^4}}(\theta)^{q^3} }{N_{q^{8}/{q^4}}(\theta )^{q^3\tau}}=1.$$
		Hence there exists $\alpha \in \F_{q^8}^*$ such that $\hat{c}=\alpha^{q-1}$. Replacing $x$ by $\alpha x$, we obtain
		\begin{equation}
			\label{eq:alpha_5}
			\left\{ (\alpha x)^{q-1}+\theta (\alpha x)^{q^5-1} : x\in \F_{q^8}^*  \right\} = \left\{ \alpha^{q-1}\left(x^{q-1}+ \alpha^{q^5-q} \theta x^{q^5-1}\right): x\in \F_{q^8}^*  \right\}.
		\end{equation}
		As $\alpha^{q^5-q} \theta=\hat{c}^{q(q^3+q^2+q+1)} \theta=\frac{\theta^{q^4+1}}{(-\theta^{q^7\tau})^q}=(-\theta^{q^7\tau})^{q^5}$, \eqref{eq:alpha_5} is equivalent to
		\begin{equation*}
			\left\{ (\alpha x)^{q-1}+\theta (\alpha x)^{q^5-1} : x\in \F_{q^8}^*  \right\} = \left\{ \hat{c}\left(x^{q-1}+ (-\theta^{q^7\tau})^{q^5} x^{q^5-1}\right): x\in \F_{q^8}^*  \right\}.
		\end{equation*}
	\end{proof}

	\begin{remark}
%		\label{re:F8_equiv}
		As pointed our in Remark \ref{re:F6_equiv} on the equivalence problem of family $\CFA$,   we also have determined the equivalence and the automorphism groups of the members in a larger family of linear sets in $\PG(1,q^8)$ which contains $\CFB$, because $f$ and $g$ in Theorem \ref{th:LP_n_8} and Theorem \ref{th:Aut_8} are not required to be scattered.
	\end{remark}
	\section{Concluding remarks}\label{sec:5}
	In this paper, we consider the equivalence of the members in the family of maximum scattering linear sets constructed by Csjab\'ok, Marino, Montanucci and Zullo in \cite{csajbok_newMSLS_2018,marino_classes_2020} and the same problem for two families $\CFA$ and $\CFB$ constructed by Csjab\'ok, Marino, Polverino and Zanella in \cite{csajbok_classes_2018}. Their automorphism groups are determined, respectively. 
	
	For the Csjab\'ok-Marino-Montanucci-Zullo family, it is proved that there is only one element in it up to $\PGamL$-equivalence, and its automorphism group is determined. For the Csjab\'ok-Marino-Polverino-Zanella family, according to the result in \cite{csajbok_classes_2018}, we only need to consider the linear set equivalence problem for $s=1$ in $\CF_{2m}$ for which the necessary and sufficient conditions for their $\PGamL$-equivalence are given. The approach adopted in this paper for the maximum scattered linear set can be used to solve the equivalence problems of other linear sets which will be considered in our future work.
	
	\section*{Acknowledgment}
	This work is partially supported by the Natural Science	Foundation of Hunan Province (No.\ 2019RS2031) and the Training Program for Excellent Young  Innovators of Changsha (No.\ kq2106006).

\end{document}